\documentclass[11pt,reqno]{amsart}

\usepackage{amsmath,amsthm,amssymb}
\usepackage{amssymb}

\usepackage{graphics,graphicx}
\usepackage{hyperref}

\usepackage[usenames, dvipsnames]{xcolor}

\usepackage{mathtools}
\mathtoolsset{showonlyrefs}

\usepackage[square,sort,comma,numbers]{natbib}
\usepackage{fancyvrb}
\usepackage{enumitem}

\definecolor{darkblue}{rgb}{0.0,0.0,0.3}
\hypersetup{colorlinks,breaklinks,
  linkcolor=darkblue,urlcolor=darkblue,
anchorcolor=darkblue,citecolor=darkblue}

\usepackage{booktabs}
\usepackage{threeparttable}

\theoremstyle{plain}
\newtheorem{theorem}{Theorem}[section]
\newtheorem*{theorem*}{Theorem}
\newtheorem{lemma}[theorem]{Lemma}
\newtheorem{proposition}[theorem]{Proposition}
\newtheorem*{proposition*}{Proposition}
\newtheorem{corollary}[theorem]{Corollary}
\newtheorem*{corollary*}{Corollary}

\theoremstyle{definition}
\newtheorem{remark}[theorem]{Remark}

\numberwithin{equation}{section}

\renewcommand{\Im}{\operatorname{Im}}
\renewcommand{\Re}{\operatorname{Re}}

\DeclareMathOperator{\R}{\mathbb{R}}
\DeclareMathOperator*{\Res}{Res}

\title{Second Moments in the Generalized Gauss Circle Problem}

\author[Hulse]{Thomas A. Hulse}
  \address{Boston College}
\author[Kuan]{Chan Ieong Kuan}
  \address{Sun Yat-Sen University}
\author[Lowry-Duda]{David Lowry-Duda}
  \address{Warwick Mathematics Institute, University of Warwick}
  \email{d.lowry@warwick.ac.uk}
  \thanks{This material is based upon work supported by the National Science
    Foundation Graduate Research Fellowship Program under Grant No. DGE 0228243.
    David also gratefully acknowledges support from EPSRC Programme Grant
    EP/K034383/1 LMF:\ L-Functions and Modular Forms.}
\author[Walker]{Alexander Walker}
  \address{Rutgers University}
  \thanks{This material is based upon work supported by the National Science Foundation under Grant No. DMS-1440140 while two of the authors were in residence at the Mathematical Sciences Research Institute in Berkeley, California, during the Spring 2017 semester.}

\begin{document}

\begin{abstract}
  The generalized Gauss circle problem concerns the lattice point discrepancy
  of large spheres.
  We study the Dirichlet series associated to $P_k(n)^2$, where
  $P_k(n)$ is the discrepancy between the volume of the $k$-dimensional sphere
  of radius $\sqrt{n}$ and the number of integer lattice points contained in
  that sphere.
  We prove asymptotics with improved power-saving error terms for smoothed
  sums, including $\sum P_k(n)^2 e^{-n/X}$ and the Laplace transform
  $\int_0^\infty P_k(t)^2 e^{-t/X}dt$, in dimensions $k \geq 3$.
  We also obtain main terms and power-saving error terms for the sharp sums
  $\sum_{n \leq X} P_k(n)^2$, along with similar results for the sharp integral
  $\int_0^X P_3(t)^2 dt$.
  This includes producing the first power-saving error term in mean square for
  the dimension-three Gauss circle problem.
\end{abstract}

\maketitle

\section{Introduction}

Let $r_k(m)$ denote the number of integer $k$-tuples $(n_1,n_2,\ldots,n_k)$ such that $n_1^2+\cdots + n_k^2 =m$, and let $S_k(n)$ denote the sum of $r_k(m)$ for $m \leq n$,
\[
  S_k(n) = \sum_{0 \leq m \leq n} r_k(m).
\]
Geometrically, $S_k(n)$ counts the number of lattice points in $\mathbb{Z}^k$ contained within $B_k(\sqrt{n})$, the $k$-dimensional sphere of radius $\sqrt{n}$.
Let $V_k$ denote the volume of $B_k(1)$, the $k$-sphere of radius $1$.
It is intuitively clear that $S_k(n) \sim \mathrm{Vol}(B_k(\sqrt{n})) = V_k n^{k/2}$ as $n \to \infty$.

To describe this asymptotic more precisely, set
\begin{equation}\label{s_k_def}
  S_k(n) = V_k n^{k/2} + P_k(n).
\end{equation}
In the $k=2$ case, estimation of $P_k(n)$ is the famous \emph{Gauss circle problem}.
Here, Gauss established $P_2(n) = O(\sqrt{n})$ by relating $P_2(n)$ to the area of a narrow annulus enclosing the boundary of $B_2(\sqrt{n})$~\cite{IvicSurvey2006}.

For general $k \geq 2$, the pursuit of a minimal exponent $\alpha_k$ for which $P_k(n) = O(n^{\alpha_k +\epsilon})$ for any $\epsilon > 0$ is now known as the \emph{generalized Gauss circle problem}.
Gauss' geometric argument readily generalizes to show that $\alpha_k \leq (k-1)/2$, but $\Omega$-type results (see~\cite{IvicSurvey2006} for a survey) support the conjecture that
\begin{align} \label{eq:conjectured_exponents}
  \alpha_k = \begin{cases} \frac{1}{4}, \qquad& k = 2 \\ \frac{k}{2}-1, \qquad &k>2 \end{cases}
\end{align}
are the true sizes.
For $k \geq 4$, this conjecture is known to be true, and for $k \geq 5$ the order of growth of $P_k(n)$ is known (up to constants), as described in~\cite{Kratzel2000}.

Far less is known in the case $k \leq 3$.
In the case $k=2$, the first improvement on Gauss' result is due to Sierpi\'nski~\cite{Sierpinski06}, who established $P_2(n) = O(n^\frac{1}{3})$ using Poisson summation and the theory of exponential sums.
Incremental progress has led to Huxley's \emph{discrete Hardy-Littlewood method}~\cite{Huxley03} and the result $P_2(n) = O\left( n^{131/416+\epsilon}\right)$.
A recent preprint of Bourgain and Watt~\cite{BourgainWatt2017} proposes an improvement of this result to $P_2(n)=O(n^{517/1648+\epsilon})$.

Notable progress in dimension $k=3$ includes Landau's result $P_3(n) = O(n^{3/4})$~\cite{Landau1919} and a long series of results due to Vinagradov culminating in $P_3(n) = O(n^{2/3} (\log n)^6)$~\cite{Vinogradov1963}.
The current best result is due to Heath-Brown~\cite{HeathBrown99}, who obtained
\[
	P_3(n) = O\left(n^{\frac{21}{32}+\epsilon}\right).
\]

Some of the best evidence for the conjectured exponents~\eqref{eq:conjectured_exponents} in the generalized Gauss circle problem is given by \emph{mean square} results describing
\[
  \int_0^X \left(P_k(x) \right)^2 \, dx.
\]
In dimension $k=2$, the earliest result is due to Landau~\cite[p. 250-263]{Landau69}, who showed that
\[
 \int_0^X \left(P_2(x) \right)^2 \, dx = c_2 X^{3/2} +O(X^{1+\epsilon}).
\]
  The best result at present is due to Lau and Tsang~\cite{LT09}, who proved
\[
  \int_0^X \left(P_2(x) \right)^2 \, dx = \frac{1}{3\pi^2} \sum_{n=1}^\infty \frac{r^2_2(n)}{n^{3/2}} X^{3/2} + O \left( X\log X \log \log X\right).
\]

In the case $k=3$, a long-standing result of the above form was due to Jarn\'{i}k~\cite{Jarnik40}, who established
\begin{equation}\label{jarnik-mean-square}
  \int_0^X \left(P_3(x) \right)^2 \, dx = c_3 X^2 \log X + O \left(X^2 (\log X )^{1/2} \right)
\end{equation}
for some $c_3>0$ using the Hardy-Littlewood method.
This error was more recently improved to $O(X^2)$ by Lau~\cite{Lau1999}.
For $k \geq 4$, Jarn\'{i}k further proved mean square results with power-savings error terms of the form
\begin{equation}\label{jarnik-mean-square-high-dim}
  \int_0^X \left(P_k(x) \right)^2 \, dx = c_k X^{k-1} + O \left( g(X) \right),
\end{equation}
with
\begin{equation*}
  g(X) = \begin{cases}
    X^{\frac{5}{2}}\log X & \text{if } k = 4, \\
    X^{3}\log^2 X & \text{if } k = 5, \\
    X^{k-2} & \text{if } k > 5.
  \end{cases}
\end{equation*}
The relatively large error term in dimension three suggests that this case is the most mysterious and least understood.  For $k>5$, these results are optimal, while for $k\leq 5$ these bounds may be improved and it may be possible to extract additional  lower order terms.
More detail on progress towards the generalized Gauss circle problem and its many cousins can be found in the excellent survey~\cite{IvicSurvey2006}.

In this paper, we consider mean square estimates for the generalized Gauss circle problem, focusing on the cases $k>2$.  Our first result is a mean square estimate with exponential smoothing.
\begin{theorem}\label{theorem:main_smooth}
  For $k \geq 3$ and any $\epsilon > 0$,
  \begin{equation}
    \begin{split}
      \sum_{n = 1}^\infty P_k(n)^2 e^{-n/X} &= \delta_{[k=3]} C'_3  X^{k-1} \left(\log X +1-\gamma\right)+ C_k \Gamma(k-1)X^{k-1}  \\
      &\quad + \delta_{[k = 4]} C'_4 \Gamma(k-\tfrac{3}{2})X^{k - \frac{3}{2}} + O_\epsilon(X^{k-2 + \epsilon}),
    \end{split}
  \end{equation}
  where $C_k$, $C_3'$, and $C_4'$ are explicit constants, and
  \begin{equation*}
    \delta_{[k=n]} = \begin{cases}
      0 & \text{if } k \neq n, \\
      1 & \text{if } k = n
    \end{cases}
  \end{equation*}
  is a Kronecker delta indicator function.
\end{theorem}

\begin{remark}\label{rem:smooth_coefficients}
  The coefficients $C_3'$, $C_4'$, and $C_k$ ($k \geq 4$) are given by
  \begin{align*}
    C_3' &= \frac{\pi^2}{3\zeta^{(2)}(3)}, \qquad C_4' =\frac{16(9\sqrt{2}-8)\zeta(\frac{1}{2})\zeta(\frac{3}{2})^2\zeta(\frac{5}{2})}{7 \pi^2 \zeta(3)},\\
    C_k  &= \frac{k^2}{24}V_k^2 + \frac{\pi^k \zeta(k-2)}{12\,\Gamma(\frac{k}{2})^2\zeta^{(2)}(k)}\left(1+2^{3-k}\right).
  \end{align*}
The size of the main term in this result matches Jarn\'{i}k's mean square estimate~\eqref{jarnik-mean-square} when $k=3$, but by smoothing we expose an additional main term and a significant separation between the main terms and error term.
  An expression for the constant $C_3$ involves coefficients from the Laurent expansion of an $L$-function, and is harder to state exactly. Numerical approximation suggests that $C_3 \approx 10.6$.

  For $k > 3$, it is possible to reduce the error term to $O_\epsilon(X^{k - 2 + \frac{3-k}{2} + \epsilon})$, although this introduces additional main terms with coefficients that are explicit but hard to compute.
  Due to a line of spectral poles in the Dirichlet series $D(s, P_k \times P_k)$, which we will define below, we believe this result is the best smooth result possible.
\end{remark}

The smoothed second moment in Theorem~\ref{theorem:main_smooth} can be thought of as a discrete Laplace transform.
In~\cite{Ivic2001}, Ivi\'{c} proved that
\begin{equation*}
  \int_0^\infty P_2(t)^2 e^{-t/X} dt = c X^{\frac{3}{2}} - X + O(X^{\frac{2}{3} + \epsilon})
\end{equation*}
for a known constant $c$, which can be thought of as a normal continuous Laplace transform of the lattice point discrepancy in dimension two.
As an application of Theorem~\ref{theorem:main_smooth}, we are able to prove a very strong result concerning the Laplace transform for dimensions $k \geq 3$.

\begin{theorem}\label{theorem:laplace}
For any $\epsilon > 0$, the smoothed second moment of the lattice point discrepancy for dimension $k \geq 3$ is given by
  \begin{align*}
    \int_0^\infty P_k(t)^2 e^{-t/X}dt &= \delta_{[k=3]} C_3' X^{k-1}(\log X + 1 -\gamma) + \delta_{[k=4]} C_4' \Gamma(k-\tfrac{3}{2}) X^{k-\frac{3}{2}} \\
     &\quad +C_k \Gamma(k-1)X^{k-1} - \frac{\Gamma(k-1)\pi^k}{6\Gamma(\frac{k}{2})^2} X^{k-1} + O\left(X^{k-2+\epsilon}\right),
  \end{align*}
  where the constants are the same as in Theorem~\ref{theorem:main_smooth}.
\end{theorem}

\begin{remark}
As in Theorem~\ref{theorem:main_smooth}, the techniques of this paper can be used to give further secondary terms and reduced error terms in dimensions $k>3$.
\end{remark}

 An application of Perron's formula with another smoothed sum allows us to prove our main result, an analogue of Theorem~\ref{theorem:main_smooth} with a sharp cutoff.

\begin{theorem}\label{theorem:main_sharp}
  For each $k \geq 3$ there exists a $\lambda>0$ such that
  \begin{equation*}
    \sum_{n \leq X} P_k(n)^2 = \delta_{[k=3]}X^{k-1} \left(\frac{C_3'}{2}\log X -\frac{C_3'}{4}\right) + \frac{C_k}{k-1} X^{k-1} +O_\lambda(X^{k-1-\lambda}),
  \end{equation*}
  where $C_3'$ and $C_k$ ($k \geq 3)$ are the same constants as in Theorem~\ref{theorem:main_smooth}.
\end{theorem}

Theorem~\ref{theorem:main_sharp} resembles the smoothed result (Theorem~\ref{theorem:main_smooth}) up to constants, although the error bound is worse.
Notice that in dimension $k=3$, Theorem~\ref{theorem:main_sharp} exhibits a second main term and additional power-savings in the error term.

The sum in Theorem~\ref{theorem:main_sharp} is closely related to the mean square results~\eqref{jarnik-mean-square} and~\eqref{jarnik-mean-square-high-dim}.
However, the two results differ in that Jarn\'{i}k considers an integral over $[0,X]$, while we consider a sum of $P_k(n)$ over integral values up to $X$.
For arithmetic applications, we believe that the sum is a more natural object of study than the integral.
But as a corollary to Theorem~\ref{theorem:main_sharp}, we are able to strengthen Jarn\'{i}k and Lau's mean square estimates  given in~\eqref{jarnik-mean-square}.

\begin{theorem}\label{theorem:main_jarnik}
  There exists $\lambda>0$ such that
  \begin{equation*}
    \int_0^X \big(P_3(x)\big)^2 dx = \frac{C_3'}{2} X^2 \log X + \left(\frac{C_3}{2}-\frac{C_3'}{4}-\frac{\pi^2}{3}\right)X^2  + O_\lambda\left(X^{2-\lambda}\right),
  \end{equation*}
  where $C_3'$ and $C_3$ are the same constants as in Theorem~\ref{theorem:main_smooth}.
\end{theorem}

\section*{Acknowledgements}

We would like to thank Jeff Hoffstein, Sam Chow, Frank Thorne, and Werner Georg
Nowak for their helpful discussions and kind remarks.
We would also like to thank Aleksandar Ivi\'{c} for suggesting that we also
consider the Laplace transform, ultimately leading to our
Theorem~\ref{theorem:laplace}.

With gratitude, the first author acknowledges support from previous employment, during which time this paper was written and revised, at Colby College in Waterville, Maine and Morgan State University in Baltimore, Maryland.

The third author gratefully acknowledges support from EPSRC Programme Grant
EP/K034383/1 LMF:\ L-Functions and Modular Forms.

We would also like to thank the National Science Foundation.
The third author was partially supported by the National Science Foundation
Graduate Research Fellowship Program under Grant No. DGE 0228243.
The third and fourth authors were partially supported by the National
Science Foundation under Grant No. DMS-1440140 while in residence at the
Mathematical Sciences Research Institute in Berkeley, California, during the
Spring 2017 semester.

\section*{Description of Methodology and Outline of Paper}

We approach this problem by understanding the analytic properties of the Dirichlet series associated to $S_k(n)^2$ and $P_k(n)^2$, defined by
\[
  D(s,S_k \times S_k) = \sum_{n=1}^\infty \frac{S_k(n)^2}{n^{s+k}}, \qquad D(s,P_k \times P_k) = \sum_{n=1}^\infty \frac{P_k(n)^2}{n^{s+k-2}}.
\]
Note that the $k$ and $k-2$ in the exponents serve to normalize the Dirichlet series to converge absolutely for $\Re s > 1$, based on known mean square results.
These two Dirichlet series are closely related to the series studied by the authors in~\cite{HulseKuanLowryDudaWalker17, HKLDWshortsums}, in which meromorphic continuations were given and studied for the Dirichlet series
\begin{equation*}
  \sum_{n \geq 1} \frac{S_f(n)^2}{n^s},
\end{equation*}
where $S_f(n) = \sum_{m \leq n} a(m)$ are partial sums of the coefficients of a $\mathrm{GL}(2)$ cusp form $f(z) = \sum a(n) e(nz)$. Indeed, the techniques and analysis in this paper build on the methodology introduced to study the cusp form case.

In \S\ref{sec:decomposition}, we show that the meromorphic properties of $D(s, P_k \times P_k)$ can be understood from the properties of $D(s, S_k \times S_k)$, and vice versa.
We then decompose $D(s, S_k \times S_k)$ into \emph{diagonal} and \emph{off-diagonal} pieces.
In \S\ref{sec:mero} and \S\ref{sec:Wk_analytic} we prove that the pieces of $D(s, S_k \times S_k)$ have meromorphic continuations to the complex plane.
This analysis culminates in Theorem~\ref{theorem:DsPkPk},  which states that $D(s, S_k \times S_k)$ and $D(s, P_k \times P_k)$ have meromorphic continuation to the plane.

As in~\cite{HulseKuanLowryDudaWalker17}, the central challenge is determining the analytic behavior of the \emph{off-diagonal}, which involves the shifted convolution sum
\begin{equation*}
  Z_k(s,w) = \sum_{h \geq 1} \sum_{n \geq 0} \frac{r_k(n+h)r_k(n)}{(n+h)^{s + \frac{k}{2} - 1} h^w}.
\end{equation*}
Heuristically, this multiple Dirichlet series can be obtained from a Petersson inner product,
\begin{equation*}
  \big \langle \lvert \theta^k \rvert^2 \Im(\cdot)^{\frac{k}{2}}, P_h(\cdot, \overline{s})\big \rangle,
\end{equation*}
where $P_h(z,s)$ is a Poincar\'e series and $\theta(z) = \sum_{n \in \mathbb{Z}} e^{2\pi i n^2 z}$ is the standard theta function.
In contrast to the cusp form case, however, $\theta(z)$ has moderate growth, complicating the spectral analysis of the inner product.
Thus it is necessary to modify $\lvert \theta^k \rvert^2$ to remove this growth.
In~\S\ref{sec:meromorphic_Zk} we subtract appropriate linear combinations of Eisenstein series evaluated at specific values such that the resulting function is square-integrable.

With this modification, in \S\ref{sec:secondmoment} we are able to use an inverse Mellin transform to extract information out of the meromorphic properties of $D(s, S_k \times S_k)$ and to prove the asymptotic behavior for the smoothed sum in Theorem~\ref{theorem:main_smooth}.
In particular, we are able to show that $D(s, S_k \times S_k)$ has polynomial growth in vertical strips.

Similar techniques are used to produce a sharp second moment in \S\ref{sec:sharp_second_momentv2}.
This is achieved by proving a weak short-interval estimate and using a Perron integral.

In \S\ref{sec:laplace}, we apply Theorem~\ref{theorem:main_smooth} to prove Theorem~\ref{theorem:laplace}, our estimate for the Laplace transform of $P_k(t)^2$.  The sum in Theorem~\ref{theorem:main_smooth} can be considered as an integral of a step function, and we study the difference between this integral and the continuous Laplace transform.

We apply similar techniques in \S\ref{sec:super_jarnik} to prove our final result, a refinement of Jarn\'{i}k's dimension three mean square result~\eqref{jarnik-mean-square}.
Known bounds for $P_3(n)$ quickly reduce our study to bounds for the cross term
\[
  \sum_{n \leq X} P_3(n) n^\frac{1}{2}.
\]
We extract a main term and power-savings error for this sum using the meromorphic properties of the Dirichlet series with coefficients $P_3(n)$ and an integral transform.

\section*{Directions for Further Research}

As presented here, Theorems~\ref{theorem:main_sharp}
and~\ref{theorem:main_jarnik} show that there are two main terms and a
power-saving error term in dimension three mean square estimates, but
we do not state the size of the power-savings in the error.
In forthcoming work, the authors will analyze the growth properties of the Dirichlet series $D(s, S_k \times S_k)$ and $D(s, P_k \times P_k)$ and identify the size of the power-savings.
In close analogy to~\cite{HKLDWshortsums}, the analysis is delicate and the largest obstacle is obtaining a nuanced understanding of the growth properties of the Petersson inner product $\langle \lvert \theta \rvert^{2k} y^{\frac{k}{2}}, \mu_j \rangle$ for Maass forms $\mu_j$.
Heuristically, the authors believe that a careful analysis based on the methods of this paper would lead to $\lambda = \frac{1}{5} - \epsilon$ in Theorem~\ref{theorem:main_sharp} (in dimension $k=3$) and Theorem~\ref{theorem:main_jarnik}, for any $\epsilon >0$.
Improved techniques for handling the contributions from Maass forms would lead to better bounds.
It is not clear what the optimal error bound should be.

The methodology used to prove Theorem~\ref{theorem:main_sharp} focused on the
dimension three case, as this is the least understood.  It may be possible to
use the meromorphic properties of $D(s, S_k \times S_k)$ for $k \geq 4$ to
prove improved estimates for higher dimensions as well.  This is especially
interesting in dimension four, as the smooth second moment in
Theorem~\ref{theorem:main_smooth} suggests the existence of a second main term
in the sharp second moment of $P_4(n)$ which we have not been able to verify.

It is possible to modify the techniques of this paper to approach the
classical Gauss circle problem in two dimensions, or to understand the lattice
point discrepancy problem for general ellipsoids.
Studying the meromorphic properties of $D(s, P_2 \times P_2)$ using the
methodology of this paper should give new insight on the Gauss circle problem.
The authors examine $D(s, P_2 \times P_2)$ and how it differs from the
Dirichlet series associated to the Gauss circle problems in higher dimensions
in the forthcoming paper~\cite{HKKL17}.

\section{Decomposition of $D(s,S_k \times S_k)$}\label{sec:decomposition}

Note that $P_k(n)^2$ and $S_k(n)^2$ are related by the formula
\begin{equation}\label{eq:Pksquared_equals_Sksquared}
  P_k(n)^2 = S_k(n)^2 - 2 V_k n^\frac{k}{2} S_k(n)+V_k^2 n^{k}.
\end{equation}
This relationship induces a relationship between $D(s, P_k \times P_k)$ and $D(s, S_k \times S_k)$, described explicitly in the following proposition.

\ \

\begin{proposition}~\label{prop:DsPkSquared_vs_DsSkSquared}
  The Dirichlet series $D(s, P_k \times P_k)$ is related to $D(s, S_k \times S_k)$ through the equality
  \begin{align} \label{eq:DsPkSquared_equals_DsSkSquared}
    D(s,P_k \times P_k) &= D(s-2,S_k \times S_k) + V_k^2 \zeta(s-2)  \notag\\
      &\quad - 2V_k \zeta(s+\tfrac{k}{2}-2) -2V_k L(s-1,\theta^k)  \\
      &\quad + \frac{i V_k}{\pi} \int_{(\sigma)} L(s-1-z,\theta^k)\zeta(z) \frac{\Gamma(z)\Gamma(s+\frac{k}{2}-2-z)}{\Gamma(s+\frac{k}{2}-2)}\,dz, \notag
  \end{align}
when $\sigma > 1$ and $\Re s>\sigma$,  where $L(s, \theta^k)$ is the normalized $L$-function
  \begin{equation*}
    L(s, \theta^k) := \sum_{n \geq 1} \frac{r_k(n)}{n^{s + \frac{k}{2} - 1}}
  \end{equation*}
  associated to the $k$-th power of the theta function $\theta(z) = \sum_{n \in \mathbb{Z}} e^{2\pi i n^2 z}$.
\end{proposition}

Here and throughout this paper, we use the common notation
\begin{equation*}
  \frac{1}{2\pi i} \int_{(\sigma)} f(z) \, dz = \frac{1}{2\pi} \int_{-\infty}^\infty f(\sigma + it) \, dt.
\end{equation*}

\begin{proof}

  We begin with~\eqref{eq:Pksquared_equals_Sksquared}, divide each term by $n^{s + k - 2}$, and sum over $n \geq 1$.
  The left-hand side and first term on the right-hand side are immediate from the definitions of $D(s, P_k \times P_k)$ and $D(s-2, S_k \times S_k)$, respectively.
  Similarly, the third term on the right-hand side is immediately recognizable as $V_k^2 \zeta(s-2)$.

  For the second term, note that
  \begin{equation*}
    S_k(n) = \sum_{m = 0}^n r_k(m) = 1 + r_k(n) + \sum_{m = 1}^{n-1} r_k(m).
  \end{equation*}
  Multiplying by $n^{\frac{k}{2}}$, dividing by $n^{s+k-2}$, and summing over $n \geq 1$ yields
  \begin{equation*}
    \zeta(s+\tfrac{k}{2}-2) + \sum_{n \geq 1} \frac{r_k(n)}{n^{s + \frac{k}{2} - 2}} + \sum_{\substack{n \geq 1 \\ 0 < m < n}} \frac{r_k(m)}{n^{s + \frac{k}{2} - 2}}.
  \end{equation*}
  Swapping the order of summation in the final sum and writing $n = m + h$ shows that
  \begin{equation} \label{eq:Skseries}
    \sum_{n=1}^\infty \frac{S_k(n)}{n^{s+\frac{k}{2} - 2}}
    = \zeta(s + \tfrac{k}{2} - 2) + L(s-1, \theta^k) + \sum_{m,h \geq 1} \frac{r_k(m)}{(h+m)^{s + \frac{k}{2} - 2}}.
  \end{equation}
  We decouple $m$ and $h$ in the last sum with the identity
  \begin{equation}\label{eq:MellinBarnesIdentity}
    \frac{1}{(m+h)^s} = \frac{1}{2\pi i} \int_{(\sigma)} \frac{1}{m^{s-z} h^z} \frac{\Gamma(z) \Gamma(s-z)}{\Gamma(s)} dz, \quad (\sigma > 0, \Re s >\sigma)
  \end{equation}
  which follows from the Barnes integral 6.422(3) of~\cite{GradshteynRyzhik07}.
  For $\sigma > 1$, the $h$ sum now converges absolutely and can be collected into a single $\zeta(z)$, and for $\Re s$ sufficiently large the $m$ sum can be collected into $L(s-1-z, \theta^k)$.
  Multiplication by $-2V_k$ identifies this with the second term in~\eqref{eq:Pksquared_equals_Sksquared}, and simplification completes the proof.
\end{proof}

Through~\eqref{eq:DsPkSquared_equals_DsSkSquared} it is possible to pass analytic information from $D(s,S_k \times S_k)$ to $D(s,P_k \times P_k)$, and vice versa.
To understand the meromorphic continuation of $D(s, S_k \times S_k)$, we first decompose the Dirichlet series $D(s, S_k \times S_k)$ into a sum of simpler functions.
Our methodology is a variant of the methodology used in the proof of Proposition~3.1 in~\cite{HulseKuanLowryDudaWalker17} and builds on the proof of the previous proposition, albeit with the added wrinkle of including shifted sums in the decomposition.

\begin{proposition}\label{prop:dirichlet}
  The Dirichlet series associated to $S_k(n)^2$ decomposes into
  \begin{equation}\label{s_basic_decomp}
    \begin{split}
      D(s,S_k \times S_k) &= \zeta(s+k) + W_k(s) \\
      &\quad + \frac{1}{2\pi i} \int_{(\sigma)} W_k(s-z) \zeta(z) \frac{\Gamma(z)\Gamma(s+k-z)}{\Gamma(s+k)}dz
    \end{split}
  \end{equation}
  for $\Re s > 2$ and $1 < \sigma < \Re (s - 1)$, in which
  \begin{align*}
    W_k(s) &= \sum_{n=1}^\infty \frac{r_k(n)^2}{n^{s+k}} + 2 Z_k(s+\tfrac{k}{2}+1,0), \\
    Z_k(s,w) &= \sum_{h\geq 1}\sum_{n \geq 0} \frac{r_k(n+h)r_k(n)}{(n+h)^{s+\frac{k}{2}-1} h^w}.
  \end{align*}
  Here $Z_k(s,w)$ converges locally normally for $\Re s > 1+\frac{k}{2}$ and $\Re w \geq 0$.
\end{proposition}

\begin{proof}
We may write
\begin{align*}
  S_k(n)^2 &= \sum_{m \leq n} \sum_{\ell\leq n} r_k(m) r_k(\ell) = \sum_{m \leq n} r_k(m)^2 + 2\!\! \sum_{\ell < m \leq n} r_k(m) r_k(\ell) \\
  &= 1+r_k(n)^2 + \!\!\! \sum_{0<m<n} \! r_k(m)^2 +2 \! \sum_{m< n} r_k(m)r_k(n) + 2\!\!\!\sum_{\ell< m <n} \!\! r_k(m)r_k(\ell).
\end{align*}
In the second line, we separated out the terms in which $m=n$.

Dividing by $n^{s+k}$ and summing over $n \geq 1$ gives
\begin{align*}
  D(s, S_k \times S_k) = \sum_{n  = 1}^\infty \frac{1}{n^{s+k}}
  &+ \bigg( \sum_{n = 1}^\infty \frac{r_k(n)^2}{n^{s+k}} + 2  \sum_{\substack{n \geq 1 \\ m < n}} \frac{r_k(m)r_k(n)}{n^{s + k}} \bigg) \\
  &+ \bigg( \sum_{\substack{n \geq 1 \\ 0 < m < n}} \frac{r_k(m)^2}{n^{s + k}} + 2 \sum_{\substack{n \geq 1 \\ \ell < m < n}} \frac{r_k(m)r_k(\ell)}{n^{s+k}}\bigg).
\end{align*}
We recognize the first term as a zeta function.
The second and third terms represent the diagonal and off-diagonal (resp.) parts of a double summation, and we analyze them together.
Swapping the order of summation and writing $n = m + h$ allows us to write the third term as
\begin{equation*}
  2  \sum_{\substack{n \geq 1 \\ m < n}} \frac{r_k(m)r_k(n)}{n^{s + k}} = 2 \sum_{\substack{m \geq 0 \\h\geq 1}} \frac{r_k(m+h)r_k(m)}{(m+h)^{s+k}}.
\end{equation*}
We now recognize the second and third terms as $W_k(s)$.

The fourth and fifth terms are also closely related.
Writing $n = m+h$ and swapping the order of summation
allows us to write
\begin{equation*}
  \sum_{\substack{n \geq 1 \\ 0 < m < n}} \frac{r_k(m)^2}{n^{s+k}} + 2 \sum_{\substack{n \geq 1 \\ \ell < m < n}} \frac{r_k(m)r_k(\ell)}{n^{s+k}}
  = \sum_{\substack{h \geq 1 \\ m \geq 1}} \frac{r_k(m)^2}{(m+h)^{s+k}} + \sum_{\substack{h \geq 1 \\ m \geq 1 \\ \ell < m}} \frac{r_k(m) r_k(\ell)}{(m + h)^{s+k}}.
\end{equation*}
Notice that this pair of sums is exactly the same as the pair of sums in $W_k(s)$, except that the denominators are shifted by $h$ and there is an additional $h$ sum.
We decouple the $h$ from $m$ by using the Barnes integral identity~\eqref{eq:MellinBarnesIdentity} again.
For $\sigma > 1$, the $h$ sum converges absolutely and can be collected into a zeta function.
Simplification completes the proof of~\eqref{s_basic_decomp}.

To see that $Z_k(s,w)$ converges locally normally in the range specified, it
suffices to show that
\[
  Z_k(s,0) = \sum_{h \geq 1} \sum_{n \geq 0}
  \frac{r_k(n+h)r_k(n)}{(n+h)^{s+\frac{k}{2}-1}} = \sum_{m \geq 1}
  \frac{r_k(m)}{m^{s+\frac{k}{2}-1}} \sum_{\ell < m} r_k(\ell)
\]
which converges absolutely for $\Re s > 1+\frac{k}{2}$, following from the estimate $S_k(m) = O(m^\frac{k}{2})$ and
absolute convergence of $L(s,\theta^k)$ in $\Re s > 1$. Indeed, by positivity we have that $Z_k(\sigma_1,0)>Z_k(\sigma_2,0)$ when $\sigma_2>\sigma_1>1+\frac{k}{2}$, and that $Z_k(\Re s,0)\geq |Z_k(s,0)|$ and so we also have local normal convergence of $Z_k(s,0)$ for $\Re s > 1+\frac{k}{2}$.
\end{proof}

\section{Meromorphic Continuation of $Z_k(s,w)$}\label{sec:meromorphic_Zk}

In this section we follow a construction method analogous to that in~\cite{HoffsteinHulse13, HulseKuanLowryDudaWalker17}, and we adapt the notation there.
We seek to understand
\[
  Z_k(s,w) =\sum_{h\geq 1}\sum_{m \geq 0} \frac{r_k(m+h)r_k(m)}{(m+h)^{s+\frac{k}{2}-1} h^w}
\]
by first fixing a single $h$ and recognizing the remaining sum over $m$ as a Petersson inner product of Poincar\'e series with an appropriate modular form, namely
\begin{align} \label{eq:InnerProduct-Ph-theta}
  \big\langle \lvert \theta^k(\cdot) \rvert^2 \Im(\cdot)^{\frac{k}{2}}, P_h(\cdot,\overline{s})\big\rangle = \int_{\Gamma_0(4) \backslash \mathcal{H}} \lvert \theta^k(z) \rvert^2 \Im(z)^{\frac{k}{2}} \overline{P_h(z, \overline{s})} \; d\mu(z),
\end{align}
in which $d\mu(z)=dxdy/y^2$ and $P_h(z, s)$ is the Poincar\'e series
\[
  P_h(z,s) = \sum_{\gamma \in \Gamma_\infty \backslash \Gamma_0(4)} \Im(\gamma z)^s e^{2\pi i h \gamma z}.
\]
By expanding the inner product~\eqref{eq:InnerProduct-Ph-theta}, we get
\begin{equation*}
  \big\langle \lvert \theta^k(\cdot) \rvert^2 \Im(\cdot)^{\frac{k}{2}}, P_h(\cdot,\overline{s})\big\rangle =  \frac{\Gamma(s+\tfrac{k}{2}-1)} {(4\pi)^{s+\frac{k}{2}-1}} D_k(s; h),
\end{equation*}
where we define
\begin{equation} \label{eq:Dksh_definition}
  D_k(s;h) = \sum_{m=0}^\infty \frac{r_k(m+h)r_k(m)}{(m+h)^{s+\frac{k}{2}-1}}
\end{equation}
for $\Re s$ sufficiently large.
Dividing by $h^w$ and summing over $h \geq 1$ recovers $Z_k(s,w)$,
\begin{equation*}
  Z_k(s,w) = \sum_{h \geq 1} \frac{D_k(s; h)}{h^w} =\frac{(4\pi)^{s+\frac{k}{2}-1}}{\Gamma(s+\tfrac{k}{2}-1)} \sum_{h \geq 1} \frac{\big\langle \lvert \theta^k(\cdot) \rvert^2 \Im(\cdot)^{\frac{k}{2}}, P_h(\cdot,\overline{s})\big\rangle}{h^w}.
\end{equation*}

We would like to understand $Z_k(s,w)$ by expressing $\langle \lvert \theta^k \rvert^2 \Im^{\frac{k}{2}}, P_h \rangle$ in a different way, by replacing $P_h$ with its spectral expansion.
However, this is complicated by the fact that $\lvert \theta^k(z) \rvert^2 \Im(z)^{\frac{k}{2}}$ is not in $L^2(\Gamma_0(4) \backslash \mathcal{H})$, so it is necessary to modify $\lvert \theta^k(z) \rvert^2 \Im(z)^{\frac{k}{2}}$ to be square integrable.
We accomplish this by subtracting Eisenstein series associated to the cusps of $\Gamma_0(4)$, chosen to cancel the polynomial growth of $\lvert \theta^k(z) \rvert^2 \Im(z)^{\frac{k}{2}}$.

\subsection{Modifying $\lvert \theta^k(z) \rvert^2 \Im(z)^{\frac{k}{2}}$ to be square integrable}

Let $E_\mathfrak{a}(z,s)$ denote the Eisenstein series attached to the cusp $\mathfrak{a}$ for the group $\Gamma_0(4)$, given by
\[
  E_\mathfrak{a}(z,s) = \sum_{\gamma \in \Gamma_\mathfrak{a} \backslash \Gamma_0(4)} \Im(\sigma_\mathfrak{a}^{-1}\gamma z)^s,
\]
where $\Gamma_\mathfrak{a} \subset \Gamma_0(4)$ is the stabilizer of the cusp $\mathfrak{a}$, and $\sigma_\mathfrak{a} \in \mathrm{PSL}_2(\R)$ satisfies $\sigma_\mathfrak{a} \infty = \mathfrak{a}$ and induces an isomorphism $\Gamma_\mathfrak{a} \cong \Gamma_\infty$ via conjugation.
These Eisenstein series have Fourier expansions, which %(as in~\cite[(1.17)]{DeshouillersIwaniec82})
can be written in the form
\begin{equation} \label{eisenstein_fourier_series}
  \begin{split}
    E_\mathfrak{a}(\sigma_\mathfrak{b} z,s) &= \delta_{[\mathfrak{a}=\mathfrak{b}]} y^s + \pi^{\frac{1}{2}} \frac{\Gamma(s-\frac{1}{2})}{\Gamma(s)} \varphi_{\mathfrak{ab}0}(s)y^{1-s}\\
    &\quad+\frac{2\pi^s y^{\frac{1}{2}}}{\Gamma(s)} \sum_{n \neq 0} \lvert n \rvert^{s-\frac{1}{2}} \varphi_{\mathfrak{ab}n}(s) K_{s-\frac{1}{2}}(2\pi \lvert n \rvert y) e(nx)
  \end{split}
\end{equation}
with known coefficients $\varphi_{\mathfrak{ab} n}(s)$.
When $\mathfrak{b} = \infty$ we will often write these coefficients as $\varphi_{\mathfrak{a}n}(s)$.
From~\eqref{eisenstein_fourier_series} and asymptotics of the $K$-Bessel function it is clear that
\begin{align}~\label{eisenstein_asymptotic}
E_\mathfrak{a}(\sigma_\mathfrak{b}z,\tfrac{k}{2}) = \delta_{[\mathfrak{a}=\mathfrak{b}]}y^{\frac{k}{2}} + \pi^{\frac{1}{2}} \frac{\Gamma(\frac{k-1}{2})}{\Gamma(\frac{k}{2})} \varphi_{\mathfrak{ab}0}(\tfrac{k}{2}) y^{1-\frac{k}{2}} +O_k\left(e^{-2\pi y}\right)
\end{align}
as $\Im z \to \infty$.
For $k \geq 3$, we conclude that $E_\mathfrak{a}(\sigma_\mathfrak{b}z,\frac{k}{2})$ vanishes as $\Im z \to \infty$ except in the case $\mathfrak{a}=\mathfrak{b}$, where it converges polynomially fast to $y^{\frac{k}{2}}$.

\begin{lemma}\label{lem:subtracting_eisenstein}
  For $k \geq 3$, the function $\mathcal{V}(z)$ given by
  \begin{equation*}
    \mathcal{V}(z):=\lvert \theta^k(z) \rvert^2 \Im(z)^{\frac{k}{2}}- E_\infty(z,\tfrac{k}{2})-E_0(z,\tfrac{k}{2}),
  \end{equation*}
  vanishes at each of the cusps of $\Gamma_0(4)$.
  Therefore $\mathcal{V}(z) \in L^2(\Gamma_0(4) \backslash \mathcal{H})$.
\end{lemma}

\begin{proof}
  We compute the growth of $\lvert \theta^k(z) \rvert^2 \Im(z)^{\frac{k}{2}}$ at the three cusps $0, \tfrac{1}{2}$, and $\infty$ of $\Gamma_0(4)$ and compare to that of the Eisenstein series.

  At the cusp $\infty$, we observe directly from the Fourier expansion that
  \begin{equation*}
    \lvert \theta^k(z) \rvert^2 \Im(z)^{\frac{k}{2}} = y^\frac{k}{2} \left(1 + O(e^{-2\pi y})\right)
  \end{equation*}
  as $\Im z \to \infty$.
Thus growth at the $\infty$ cusp is exactly cancelled by subtracting the Eisenstein series $E_\infty(z,\tfrac{k}{2})$.

  At the cusp $0$, we use $\sigma_0 = \left(\begin{smallmatrix} 0& -\frac{1}{2} \\ 2&0 \end{smallmatrix}\right)$ to compute
  \begin{align*}
    \theta \big\rvert_{\sigma_0}(z) &= (-2iz)^{-\frac{1}{2}} \theta\left(\!\left(\begin{matrix} 0 & -\frac{1}{2} \\ 2 & 0 \end{matrix} \right)z \right) = (-2iz)^{-\frac{1}{2}} \theta \left(-\frac{1}{4z}\right) \\
    &=(-2iz)^{-\frac{1}{2}}(-2iz)^{\frac{1}{2}} \theta(z) = \theta(z),
  \end{align*}
in which we've used the involution equation $\theta(-1/4z) = (-2iz)^{1/2} \theta(z)$ for the theta function.
Therefore $\lvert \theta^k(\sigma_0(z)) \rvert^2 \Im(\sigma_0 z)^\frac{k}{2} = y^\frac{k}{2}(1+O(e^{-2\pi y}))$ as $z \to \infty$, hence subtracting $E_0(z, \tfrac{k}{2})$ cancels the growth at the $0$ cusp.

To address the cusp $\tfrac{1}{2}$, we first note that $\theta(z+\tfrac{1}{2}) = 2\theta(4z)-\theta(z)$ by comparison of Fourier expansions. The functional equation of $\theta(z)$ gives
\[\theta(z+\tfrac{1}{2}) = 2 \theta(4z) - \theta(z) = (-2iz)^{-\frac{1}{2}} \left( \theta\left(\frac{-1}{16z}\right) - \theta\left(\frac{-1}{4z}\right)\right),
\]
  which converges to $0$ exponentially fast as $z \to 0$ non-horizontally in
  $\mathcal{H}$.
  Thus $\lvert \theta^k(z) \rvert^2 \Im(z)^\frac{k}{2} \to 0$ as $z\to
  \tfrac{1}{2}$ and it is not necessary to mitigate any growth at the cusp
  $\tfrac{1}{2}$.

%we use $\sigma_{1/2} = \left(\begin{smallmatrix} 1&0 \\ 2&1 \end{smallmatrix}\right)$ to compute
%  \begin{align}
%    \theta\big\rvert_{\sigma_{1/2}}(z) &= (2z+1)^{-\frac{1}{2}} \theta\left(\!\left(\begin{matrix} 1 & 0 \\ 2 & 1 \end{matrix}\right) z \right) = (2z+1)^{-\frac{1}{2}} \theta\!\left(\frac{z}{2z+1}\right) \nonumber \\
%    &=(2z+1)^{-\frac{1}{2}} \theta\!\left(\frac{-1}{-2-1/z}\right)
%    =\left(\frac{i}{2z}\right)^{\frac{1}{2}} \theta\!\left(-\frac{1}{2}-\frac{1}{4z}\right) \nonumber \\
%    &= \left(\frac{i}{2z}\right)^{\frac{1}{2}} \left(2 \theta\!\left(-\frac{1}{z}\right)- \theta\!\left(-\frac{1}{4z}\right)\right), \label{0_cusp_slash}
%  \end{align}
%  in which we've used that $\theta(z-\frac{1}{2}) = 2 \theta(4z)-\theta(z)$, as can be seen by comparing the Fourier series of each term.
%  Applying the involution equation to each theta function in~\eqref{0_cusp_slash}, we see that $\theta_{\lvert \sigma_{1/2}}(z) = \theta(\frac{z}{4}) - \theta(z)$, hence
%  \[
%    \theta(\sigma_{\frac{1}{2}} z) \Im(\sigma_{\frac{1}{2}} z)^{\frac{1}{4}} = O\left(e^{-\pi y /2}\right)
%  \]
%  as $z \to \infty$. Thus $\lvert \theta^k(z) \rvert^2 \Im(z)^\frac{k}{2} \to 0$ as $z\to \tfrac{1}{2}$ and it is not necessary to mitigate any growth at the cusp $\tfrac{1}{2}$.
\end{proof}

We will use $\mathcal{V}(z)$ in place of $\lvert \theta^k(z) \rvert^2 \Im(z)^{\frac{k}{2}}$ to derive the analytic properties of $Z_k(s,w)$.
Replacing~\eqref{eq:InnerProduct-Ph-theta} with the inner product
$\left\langle \mathcal{V}(\cdot), P_h(\cdot,\overline{s})\right\rangle
$
and performing the calculations from the start of this section yields
\begin{equation} \label{series_inner_product}
  \begin{split}
    &\frac{(4\pi)^{s+\frac{k}{2}-1}}{\Gamma(s+\frac{k}{2}-1)}\big\langle \mathcal{V},P_h(\cdot,\overline{s})\big\rangle \\
    &\qquad= D_k(s; h) - \frac{(2\pi)^k\Gamma(s-\frac{k}{2})}{\Gamma(\frac{k}{2})\Gamma(s)} \frac{\big(\varphi_{\infty h}(\tfrac{k}{2}) + \varphi_{0h}(\tfrac{k}{2})\big)}{h^{s - \frac{k}{2}}},
  \end{split}
\end{equation}
where $D_k(s;h)$ is as in~\eqref{eq:Dksh_definition}.
We note that we use~\cite[6.621(3)]{GradshteynRyzhik07} to evaluate the $y$-integral involved in expanding the inner products concerning the Eisenstein series.
Dividing by $h^w$, summing over $h \geq 1$, and rearranging yields
\begin{equation}\label{eq:Zk_equals_inner_and_eisenstein}
  \begin{split}
    Z_k(s,w) &= \frac{(4\pi)^{s+\frac{k}{2}-1}}{\Gamma(s+\frac{k}{2}-1)}\sum_{h \geq 1} \frac{\big\langle \mathcal{V},P_h(\cdot,\overline{s})\big\rangle}{h^w} \\
    &\quad + \frac{(2\pi)^k\Gamma(s-\frac{k}{2})}{\Gamma(\frac{k}{2})\Gamma(s)} \sum_{h \geq 1} \frac{\big(\varphi_{\infty h}(\tfrac{k}{2}) + \varphi_{0h}(\tfrac{k}{2})\big)}{h^{s + w - \frac{k}{2}}}.
 \end{split}
\end{equation}

\subsection{Spectral Expansion}

By Selberg's Spectral Theorem (as in~\cite[Theorem~15.5]{IwaniecKowalski04}),
the Poincar\'{e} series $P_h(z,s)$ has a spectral expansion of the form
\begin{equation} \label{spectral_expansion_poincare}
  \begin{split}
    P_h(z,s) = &\sum_j \langle P_h(\cdot,s),\mu_j \rangle \mu_j(z) \\
    &+
    \sum_\mathfrak{a} \frac{1}{4\pi}\int_{-\infty}^\infty
    \langle P_h(\cdot,s),E_\mathfrak{a}(\cdot,\tfrac{1}{2}+it)\rangle
    E_\mathfrak{a}(z,\tfrac{1}{2}+it)dt,
  \end{split}
\end{equation}
where $\mathfrak{a}$ ranges over the cusps of $\Gamma_0(4) \backslash
\mathcal{H}$, and
$\{\mu_j\}$ denotes an orthonormal basis of the residual and cuspidal spaces,
consisting of the constant form $\mu_0$ and of Hecke-Maass forms $\mu_j$ for
$L^2(\Gamma_0(4)\backslash \mathcal{H})$ with associated types
$\frac{1}{2}+it_j$.
The inner product of the Poincar\'{e} series against the constant term $\mu_0$
vanishes, so we omit further consideration of it.
We think of the sum over $j$ as the ``discrete part of the spectrum'' and the
sum of integrals of Eisenstein series as the ``continuous part of the
spectrum.''
Each Maass forms admits a Fourier expansion of the form
\begin{equation}\label{eq:maassfourier}
  \mu_j(z) = \sum_{n \neq 0} \rho_j(n) y^{\frac{1}{2}} K_{it_j}(2\pi \lvert n \rvert y)e(nx),
\end{equation}
where $e(x) = e^{2\pi i x}$, and has an associated $L$-function of the form
\begin{equation*}
  L(s, \mu_j) = \sum_{n \geq 1} \frac{\rho_j(n)}{n^s}.
\end{equation*}

In this section, we use the spectral expansion~\eqref{spectral_expansion_poincare} in the inner product in~\eqref{eq:Zk_equals_inner_and_eisenstein} to prove the following proposition.

\begin{proposition}\label{prop:Zksw_spectral_decomposition}
  For $\Re s$ sufficiently large, the shifted convolution sum $Z_k(s,w)$ can be expressed as
  \begin{align}
    Z_k(s,w) &= \frac{(2\pi)^k \Gamma(s-\frac{k}{2})}{\Gamma(\frac{k}{2})\Gamma(s)} \sum_{h=1}^\infty \frac{\left(\varphi_{0h}(\frac{k}{2})+\varphi_{\infty h}(\frac{k}{2})\right)}{h^{w+s-\frac{k}{2}}} \nonumber \\
  &\quad+ \frac{(4\pi)^\frac{k}{2}}{2}\sum_j  G(s,it_j) L(s+w-\tfrac{1}{2},\mu_j)\langle \mathcal{V},\mu_j \rangle \label{spectral_expansion_zk} \\  \nonumber
  &\quad+\frac{(4\pi)^{\frac{k}{2}}}{4\pi i} \!\sum_{\mathfrak{a}} \!\int_{(0)}\!\!\frac{G(s,z)\pi^{\frac{1}{2}+z}}{\Gamma(\frac{1}{2}+z)} \sum_{h \geq 1} \frac{\overline{\varphi_{\mathfrak{a}h}(\frac{1}{2}-z)}}{h^{s+w-\frac{1}{2}-z}}\langle \mathcal{V},E_\mathfrak{a}(\cdot,\tfrac{1}{2}-\overline{z})\rangle dz,
  \end{align}
  in which $G(s,z)$ denotes the collected gamma factors,
  \[
    G(s,z):= \frac{\Gamma(s-\frac{1}{2}+z)\Gamma(s-\frac{1}{2}-z)}{\Gamma(s+\frac{k}{2}-1)\Gamma(s)}.
  \]
  We refer to the first line of~\eqref{spectral_expansion_zk} as the ``non-spectral part,'' to the second line as the ``discrete part of the spectrum,'' and to the third line as the ``continuous part of the spectrum.''
\end{proposition}

\begin{proof}

The automorphic invariance and Fourier expansion of Maass forms can be used to
expand the inner product of $\mu_j$ against the Poincar\'e series via a
standard unfolding argument and the integral identity
\cite[6.621(3)]{GradshteynRyzhik07}.
One obtains
\[
  \big\langle P_h(\cdot,s),\mu_j\big\rangle=\frac{\overline{\rho_j(h)} \sqrt{\pi}}{(4\pi h)^{s-\frac{1}{2}}} \frac{\Gamma(s-\frac{1}{2}-it_j)\Gamma(s-\frac{1}{2}+it_j)}{\Gamma(s)}.
\]
It follows that the discrete part of the spectrum of $P_h(z,s)$ can be written as
\begin{align} \label{poincare_discrete}
\frac{\sqrt{\pi}}{(4\pi h)^{s-\frac{1}{2}} \Gamma(s)}\sum_j \overline{\rho_j(h)} \Gamma(s-\tfrac{1}{2}-it_j)\Gamma(s-\tfrac{1}{2}+it_j).
\end{align}
We have $\sup_j\{ \vert \Im t_j\vert \} =0$ as a consequence of Huxley's proof of the Selberg Eigenvalue Conjecture for Maass forms of small level~\cite{Huxley85}, which we note implies that~\eqref{poincare_discrete} is analytic in the right half-plane $\Re s > \frac{1}{2}$.

The inner product of the Poincar\'e series against the Eisenstein series $E_\mathfrak{a}(z,w)$ can similarly be computed to be
\[
  \big\langle P_h(\cdot,s), E_{\mathfrak{a}}(\cdot,w)\big\rangle
  = \frac{2 \pi^{\overline{w} + \frac{1}{2}} } {(4\pi h)^{s - \frac{1}{2}} }
  h^{\overline{w}-\frac{1}{2}} \varphi_{\mathfrak{a}h}(\overline{w})
  \frac{\Gamma(s+\overline{w}-1)\Gamma(s-\overline{w})}{\Gamma(s)\Gamma(\overline{w})},
\]
provided that $\Re s > \left\lvert \Re w-\frac{1}{2} \right\rvert + \frac{1}{2}$. With $t \in \R$ and $w=\frac{1}{2}+it$, this specializes to
\[
  \big\langle P_h(\cdot,s), E_\mathfrak{a}(\cdot, \tfrac{1}{2}+it)\big\rangle
  = \frac{2\pi^{1-it} \varphi_{\mathfrak{a} h}(\tfrac{1}{2}-it) }{(4\pi h)^{s-\frac{1}{2}}}
  \frac{\Gamma(s-\frac{1}{2}-it)\Gamma(s-\frac{1}{2}+it)}{h^{it}\Gamma(s)\Gamma(\frac{1}{2}-it)},
\]
which is valid provided that $\Re s  > \frac{1}{2}$.
Thus the continuous part of the spectrum of $P_h(z,s)$ takes the form
\begin{align} \label{poincare_continuous}
  \frac{1}{2} \sum_\mathfrak{a} \int_{-\infty}^\infty \frac{\varphi_{\mathfrak{a} h}(\tfrac{1}{2}-it)\Gamma(s-\frac{1}{2}-it)\Gamma(s-\frac{1}{2}+it)}{(4\pi h)^{s-\frac{1}{2}}(\pi h)^{it}\Gamma(s)\Gamma(\frac{1}{2}-it)} E_\mathfrak{a}(z,\tfrac{1}{2}+it)dt.
\end{align}

Substituting the discrete part of the spectrum~\eqref{poincare_discrete} and continuous part of the spectrum~\eqref{poincare_continuous} into the expansion of the Poincar\'e series~\eqref{spectral_expansion_poincare} gives
\begin{align*}
  &\left\langle \mathcal{V}, P_h(\cdot, \overline{s} ) \right\rangle = \frac{\sqrt{\pi}}{(4\pi h)^{s-\frac{1}{2}}\Gamma(s)} \!\sum_j \rho_j(h) \Gamma(s-\tfrac{1}{2}+it_j)\Gamma(s-\tfrac{1}{2}-it_j) \langle \mathcal{V},\mu_j \rangle \\
  &\hspace{9 mm}+\frac{1}{2} \sum_{\mathfrak{a}} \!\int_{-\infty}^\infty \!\!\frac{\overline{\varphi_{\mathfrak{a}h}(\frac{1}{2}-it)} \Gamma(s-\frac{1}{2}+it)\Gamma(s-\frac{1}{2}-it)}{(4\pi h)^{s-\frac{1}{2}}(\pi h)^{-it} \Gamma(s)\Gamma(\frac{1}{2}+it)} \langle \mathcal{V},E_\mathfrak{a}(\cdot,\tfrac{1}{2}+it)\rangle dt.
\end{align*}
Finally, substituting into~\eqref{eq:Zk_equals_inner_and_eisenstein} and simplifying completes the proof.
\end{proof}

\subsection{Meromorphic Continuation}
\label{sec:mero}

In order to provide the meromorphic continuation of $Z_k(s,w)$, we give the meromorphic continuation of each part of~\eqref{spectral_expansion_zk}.
We will prove the following lemma as a step towards understanding the analytic behavior of $W_k(s)$, which we study in \S\ref{sec:Wk_analytic}.

\begin{lemma}
  The shifted convolution $Z_k(s,w)$ has meromorphic continuation to $\mathbb{C}^2$.
  In particular, the specialized convolution sum $Z_k(s,0)$ has meromorphic continuation to the plane.
  For $\Re s > -\frac{1}{2}$, all poles of $Z_k(s,0)$ come from the non-spectral part (which has poles at $s = 1 + \frac{k}{2} - j$ for $j \in \mathbb{Z}_{\geq 0}$) and the continuous part of the spectrum (whose poles appear within the residual terms $\mathcal{R}^\pm_j$, as defined in \S\ref{sec:continuous_part_continuation}).
\end{lemma}

\subsubsection{Non-Spectral Part}

When $\mathfrak{b} = \infty$ and the cusp $\mathfrak{a}$ is represented in the
form $\mathfrak{a} = u/v$ with $(u,v)=1$, the exact definition of the
coefficients $\varphi_{\mathfrak{ab}h}(s)$ in~\eqref{eisenstein_fourier_series}
is given in~\cite[p. 247]{DeshouillersIwaniec82} by the formula
\begin{equation*}
  \varphi_{\mathfrak{a} h}(s)
  =
  \left(\frac{(v,4/v)}{4v}\right)^s \sum_{(\gamma, 4/v) =1}^\infty \gamma^{-2s}
  \hspace{-3 mm}
  \sum_{\substack{
    \delta (\gamma v)^*
    \\
    \gamma\delta v \equiv u v \!\!\!\!\! \mod (v^2,4)
  }}
  \! e\left(\frac{h\delta}{\gamma v}\right).
\end{equation*}
\begin{remark}
  The formula in~\cite{DeshouillersIwaniec82} has a minor error in the
  congruence condition in the sum. It is missing a factor of $v$ on the left
  (where our $v$ is $w$ in their notation).
\end{remark}
We represent the three inequivalent cusps $0, \frac{1}{2}$, and $\infty$ of
$\Gamma_0(4)$ as $1, \frac{1}{2}$, and $\frac{1}{4}$, respectively.
It is a standard exercise to compute these coefficients
(see~\cite[\S3.1]{Goldfeld06} for a similar calculation), and we find that
\begin{align*}
  \varphi_{0h}(s) &= \frac{\sigma^{(2)}_{1-2s}(h)}{4^s\zeta^{(2)}(2s)}, \qquad \varphi_{\frac{1}{2}h}(s) =  \frac{(-1)^h\sigma^{(2)}_{1-2s}(h)}{4^s\zeta^{(2)}(2s)}, \\
  \varphi_{\infty h}(s) &= \frac{2^{2-4s}\sigma_{1-2s}(\frac{h}{4}) - 2^{1-4s}\sigma_{1-2s}(\frac{h}{2})}{\zeta^{(2)}(2s)}.
\end{align*}
in which $\zeta^{(2)}(s)$ is the Riemann zeta function with its $2$-factor removed, $\sigma_\nu(h)$ is the sum of divisors function, and $\sigma_\nu^{(2)}(h)$ is the sum of odd-divisors function.
Dividing by $h^w$ and summing over $h$, we compute
\begin{equation} \label{eisenstein_coefficient_dirichlet_series}
  \begin{split}
    \sum_{h \geq 1} \frac{\varphi_{0h}(s)}{h^w} &= \frac{\zeta(w) \zeta^{(2)}(w-1+2s)}{4^t \zeta^{(2)}(2s)}, \\
    \sum_{h \geq 1} \frac{\varphi_{\frac{1}{2}h}(s)}{h^w} &= \frac{(2^{1-w}-1)\zeta(w)\zeta^{(2)}(w-1+2s)}{4^s \zeta^{(2)}(2s)},\\
    \sum_{h \geq 1} \frac{\varphi_{\infty h}(s)}{h^w} &= \frac{\zeta(w)\zeta(w-1+2s)}{2^{4s} \zeta^{(2)}(2s)} \left(\frac{1}{4^{w-1}}-\frac{1}{2^{w-1}} \right).
  \end{split}
\end{equation}
Applying these expressions to the spectral decomposition from Proposition~\ref{prop:Zksw_spectral_decomposition}, we rewrite the non-spectral part as
\begin{align*}
&\frac{\pi^k \Gamma(s-\frac{k}{2})\zeta(s+w-\frac{k}{2})\zeta(s+w+\frac{k}{2}-1)}{\Gamma(\frac{k}{2})\Gamma(s)\zeta^{(2)}(k)} \left(1+ \frac{4}{2^{2s+2w}} - \frac{4}{2^{\frac{k}{2}+s+w}}\right).
\end{align*}
This expression is analytic in the region $\Re s >k/2$ and $\Re(s+w) > 1+k/2$, and extends meromorphically to all of $\mathbb{C}^2$ with polar lines at $s+w=1+k/2$, $s+w=2-k/2$, and poles in $s$ at poles of $\Gamma(s-\tfrac{k}{2})/\Gamma(s)$.
Specializing to the case $w=0$, we note potential poles at $s=1+\frac{k}{2}-j$ for each integer $j\geq 0$.

\subsubsection{Discrete Part of the Spectrum}\label{sec:discrete_part_continuation}

The discrete part of the spectrum from~\eqref{spectral_expansion_zk} has clear meromorphic continuation induced by the meromorphic continuations of the individual $L(s, \mu_j)$.
We note that for any fixed $s$, the gamma functions in $G(s, it_j)$ give exponential decay so that the sum converges absolutely.

Note also that  $\langle \mathcal{V}, \mu_j \rangle = 0$ when $\mu_j$ is odd.
Indeed, $\lvert \theta^k(z) \rvert ^2\Im(z)^{\frac{k}{2}}$ is even and Eisenstein series are orthogonal to cusp forms.
Otherwise, if $\mu_j$ is even, we note by the functional equation of $L$-functions of even Maass forms that $L(-2m \pm it_j ,\mu_j)=0$ for any $m \in \mathbb{Z}_{\geq 0}$.
Specializing now to $w = 0$, these two observations combine to indicate that the apparent poles at $s = \frac{1}{2} \pm it_j$ do not exist.
Therefore the discrete part of the spectrum is analytic for $\Re s > - \frac{1}{2}$ and has poles at $s -\frac{1}{2} \pm it_j = -m$ for $m$ odd, $m \in \mathbb{Z}_{> 0}$.

\subsubsection{Continuous Part of the Spectrum}\label{sec:continuous_part_continuation}

The continuous part of the spectrum from~\eqref{spectral_expansion_zk} requires more nuanced analysis than the discrete part or non-spectral part, due to the interaction of independent complex variables.

For notational simplicity, we write the continuous part in the form
\begin{align} \label{continuous_part_simplified}
  \frac{(4\pi)^{\frac{k}{2}}}{4\pi i} \sum_{\mathfrak{a}} \int_{(0)}\!\frac{G(s,z)\pi^{\frac{1}{2}+z}}{\Gamma(\frac{1}{2}+z)} \zeta_\mathfrak{a}(s+w,z) \langle \mathcal{V},E_\mathfrak{a}(\cdot,\tfrac{1}{2}-\overline{z})\rangle dz,
\end{align}
in which $\zeta_\mathfrak{a}(s,z)$ is defined by
\[
  \zeta_\mathfrak{a}(s,z) = \sum_{h \geq 1} \frac{\overline{\varphi_{\mathfrak{a} h}(\frac{1}{2}-z)}}{h^{s-\frac{1}{2}-z}}.
\]
It is quickly verified using~\eqref{eisenstein_coefficient_dirichlet_series} that
\begin{equation*}
  \zeta_0(s,z)
  =
  \frac{\zeta(s-\frac{1}{2}-z)\zeta^{(2)}(s-\frac{1}{2}+z)}{2^{1+2z} \zeta^{(2)}(1+2z)}.
  %\qquad
  %\zeta_{\frac{1}{2}}(s, z) = \zeta_0(s, z) \left(\frac{2^z}{2^{s + \frac{3}{2}}} - 1 \right), \\
  %\zeta_\frac{1}{2}(s,z) = \frac{\zeta(s-\frac{1}{2}-z)\zeta^{(2)}(s-\frac{1}{2}+z)}{2^{1+2z}\zeta^{(2)}(1+2z)}\left(\frac{2^{z}}{2^{s+\frac{3}{2}}}-1\right), \\
  %\zeta_\infty(s,z) =\frac{\zeta(s-\frac{1}{2}-z)\zeta(s-\frac{1}{2}+z)}{2^{2+4z}\zeta^{(2)}(1+2z)}\left(\frac{4^z}{4^{s-\frac{3}{2}}}-\frac{2^z}{2^{s-\frac{3}{2}}}\right).
\end{equation*}
(The expressions associated to the other cusps are very similar).
It is now clear that the continuous part of the spectrum is analytic in the
region $\Re(s + w) > \frac{3}{2}$ and $\Re s > \frac{1}{2}$, and that the
integrand has apparent poles when
$s+w-\frac{1}{2} \pm z = 1$ and $s=\frac{1}{2}\pm z -j$
for $j \in \mathbb{Z}_{\geq 0}$.
It is now necessary to disentangle these poles from the integration variable.

Arguing as in~\cite[\S4.4.2]{HulseKuanLowryDudaWalker17} and~\cite{HoffsteinHulse13}, we iteratively extend the meromorphic continuation of the continuous part of the spectrum by carefully shifting lines of integration and collecting residual terms.

For small $\epsilon >0$, let  $\Re s$  lie in the interval $(\tfrac{3}{2}-\Re w, \frac{3}{2}-\Re w+\epsilon)$ and furthermore suppose $s$ is at least a distance of $2\epsilon$ from the potential poles of $G(s,z)$. We shift the $z$-contour to the right, along a contour $C$ which bends to remain in the zero-free region of $\zeta(1-2z)$ and thus avoids potential poles contributed by the inner product, $\langle \mathcal{V},E_\mathfrak{a}(\cdot,\tfrac{1}{2}-\overline{z})\rangle$. In so doing, we pass a pole at $s+w-\frac{1}{2}-z=1$ with residue
\begin{align*}
  \mathcal{R}_{1}^-:&=\frac{(4\pi)^\frac{k}{2}}{2} \!\Res_{z=s+w-\frac{3}{2}} \frac{G(s,z) \pi^{\frac{1}{2}+z}}{\Gamma(\frac{1}{2}+z)} \sum_\mathfrak{a} \zeta_\mathfrak{a}(s+w,z)\left\langle \mathcal{V},E_\mathfrak{a}(\cdot,\tfrac{1}{2}-\overline{z})\right\rangle.
  \end{align*}
The $2$-factors in $\zeta_\infty(s+w,z)$ and $\zeta_\frac{1}{2}(s+w,z)$ create zeros that cancel the pole, so the only cusp that gives a polar contribution at $z=s+w-\frac{3}{2}$ is the $0$ cusp.
Simplifying, we find that
\begin{align} \label{r_1plus}
  \mathcal{R}_1^- = -\frac{(4\pi)^\frac{k}{2}}{2 \pi^{1-s-w}} \frac{\Gamma(1-w)\Gamma(2s+w-2) \left\langle \mathcal{V},E_0(\cdot, 2- \overline{s}-\overline{w})\right\rangle}{2^{2s+2w-2}\Gamma(s)\Gamma(s+\frac{k}{2}-1)\Gamma(s+w-1)}.
\end{align}

The residue $\mathcal{R}_1^- = \mathcal{R}_1^-(s,w)$ has a straightforward
meromorphic continuation to all $\mathbb{C}^2$.  Our deformation of the contour
integral~\eqref{continuous_part_simplified} is analytic for $s$ to the right of
the contour $\frac{3}{2}- \Re w - C$ and to the left of the line
$\frac{3}{2}-\Re w +\epsilon$.
When $s$ is moved just to the left of the $\frac{3}{2}-\Re w$ line in this
region, we can shift the contour of $z$ integration back to $\Re z = 0$.
This passes over the \emph{other} pole at $s + w - \tfrac{1}{2} + z = 1$ from the other zeta function and introduces a residue
\begin{align}\label{eq:R_1_plus_def}
  \mathcal{R}_{1}^{+} &:= \frac{(4\pi)^\frac{k}{2}}{2} \!\!\Res_{z=\frac{3}{2}-s-w} \!\!\frac{G(s,z) \pi^{\frac{1}{2}+z}}{\Gamma(\frac{1}{2}+z)} \!\sum_\mathfrak{a} \zeta_\mathfrak{a}(s+w,z)
  \left\langle \mathcal{V},E_\mathfrak{a}(\cdot,\tfrac{1}{2}-\overline{z})\right\rangle
%\\
%\nonumber
%&= \frac{(4\pi)^\frac{k}{2}V}{2} \frac{G(s,\frac{3}{2}-s-w) \pi^{2-s-w}}{\Gamma(2-s-w)} \!\sum_\mathfrak{a}  \!\!\Res_{z=\frac{3}{2}-s-w} \!\! \zeta_\mathfrak{a}(s+w,z)
%  \left\langle \mathcal{V},E_\mathfrak{a}(\cdot,\overline{s}+\overline{w}-1)\right\rangle
\end{align}

The residue $\mathcal{R}_1^+$
also has a straightforward meromorphic continuation.
We note that the shifted contour integral has no further poles with $\Re(s+w) > \tfrac{1}{2}$ and $\Re s > \frac{1}{2}$.
Therefore the continuous part of the spectrum, originally defined for $\Re(s+w) > \tfrac{3}{2}$ and  $\Re s > \frac{1}{2}$, has meromorphic extension to $\Re (s+w) > \frac{1}{2}$ and  $\Re s > \frac{1}{2}$, given by
\begin{equation*}
  \frac{(4\pi)^{\frac{k}{2}}}{4\pi i} \sum_{\mathfrak{a}} \int_{(0)}\!\frac{G(s,z)\pi^{\frac{1}{2}+z}}{\Gamma(\frac{1}{2}+z)} \zeta_\mathfrak{a}(s+w,z) \langle \mathcal{V},E_\mathfrak{a}(\cdot,\tfrac{1}{2}-\overline{z})\rangle dz + \mathcal{R}_1^+ - \mathcal{R}_1^-,
\end{equation*}
where by a slight abuse of notation we claim that the two residual terms $\mathcal{R}_1^\pm(s,w)$ appear in the continuation only when $\Re(s+w) < \tfrac{3}{2}$, and with a slight variation when $\Re(s+w)=\frac{3}{2}$.

We now iterate this argument to push the meromorphic continuation of the continuous part past additional polar lines, as in~\cite[\S4, p. 481-483]{HoffsteinHulse13} or~\cite[\S4]{HulseKuanLowryDudaWalker17}.
That is, for $\Re s$ near $\tfrac{1}{2} - j$ with $j \in \mathbb{Z}_{\geq 0}$, we shift the line of integration in $z$ past a pole due to a gamma factor in the numerator of $G(s,z)$, move $s$ left past the polar line, and shift the line of integration back to the imaginary axis, passing a pole from the other gamma factor in the numerator of $G(s,z)$.
Each iteration contributes two additional residual terms with opposite signs, denoted by $\mathcal{R}_{-j}^+ - \mathcal{R}_{-j}^{-}$, in which
\begin{align*}
  \mathcal{R}^+_{-j}
  &=
  \frac{(4\pi)^\frac{k}{2}}{2} \sum_\mathfrak{a} \Res_{z=\frac{1}{2}-j-s}
  \frac{G(s,z) \pi^{\frac{1}{2}+z}}{\Gamma(\frac{1}{2}+z)} \zeta_\mathfrak{a}(s+w,z)
  \left\langle \mathcal{V}, E_\mathfrak{a}(\cdot, \tfrac{1}{2}-\overline{z})\right\rangle,
  \\
  \mathcal{R}^-_{-j}
  &=
  \frac{(4\pi)^\frac{k}{2}}{2} \sum_\mathfrak{a} \Res_{z=s+j-\frac{1}{2}}
  \frac{G(s,z) \pi^{\frac{1}{2}+z}}{\Gamma(\frac{1}{2}+z)} \zeta_\mathfrak{a}(s+w,z)
  \left\langle \mathcal{V}, E_\mathfrak{a}(\cdot, \tfrac{1}{2}-\overline{z})\right\rangle.
\end{align*}
Note that the notation $\mathcal{R}^{\pm}_{-j}$ resembles the notation for $\mathcal{R}^{\pm}_1$, but the source of the poles for $\mathcal{R}^{\pm}_{-j}$ are the gamma functions in $G(s,z)$ instead of the zeta functions in $\zeta_a(s+w, z)$.
Thus the locations of the poles in $\mathcal{R}^{\pm}_{1}$ depend on $w$ while the locations of the poles in $\mathcal{R}^{\pm}_{-j}$ do not.
Each of these residual terms has an easily understood meromorphic continuation.
In this way, we obtain the meromorphic continuation of $Z_k(s,w)$ to the entire complex plane.

\section{Analytic Behavior of $W_k(s)$}\label{sec:Wk_analytic}

In this section, we outline some of the analytic properties of $W_k(s)$.
These properties will be used in \S\ref{sec:AnalysisD_PkPk} to understand $D(s, P_k \times P_k)$.

Recall from Proposition~\ref{prop:dirichlet} that
\begin{equation} \label{eq:wks_definition_reminder}
  W_k(s) = \sum_{n \geq 1} \frac{r_k(n)^2}{n^{s+k}} + 2Z_k(s + \tfrac{k}{2} + 1, 0).
\end{equation}
We refer to the sum in~\eqref{eq:wks_definition_reminder} as the \emph{diagonal part}.  The second term, $Z_k(s,w)$, is the \emph{off-diagonal part}, which we recall decomposes into three terms we have called the \emph{non-spectral}, \emph{discrete}, and \emph{continuous parts}.

\begin{theorem}\label{theorem:Wks_mero}
The function $W_k(s)$ has meromorphic continuation to all $s \in \mathbb{C}$.
In the half-plane $\Re s > - \frac{k+3}{2}$, all but one of the poles of $W_k(s)$ occur at non-positive even integers and come from the non-spectral part
\begin{equation*}
      \mathfrak{E}_k(s) = \frac{2\pi^k \Gamma(s+1)\zeta(s+1)\zeta(s+k)}{\Gamma(\frac{k}{2})\Gamma(s+\frac{k}{2}+1)\zeta^{(2)}(k)} \left(1+ \frac{1}{2^{2s+k}} - \frac{1}{2^{s+k-1}}\right).
    \end{equation*}

The function $W_k(s)$ has an additional pole at $s=-\frac{k+1}{2}$.  When $k>3$, this pole is simple and has residue
  \begin{align*}
      \Res_{s = -\frac{k+1}{2}} W_k(s) = (4\pi)^{\frac{k}{2}} \frac{\langle \mathcal{V}, E_0(\cdot, \frac{3}{2}) \rangle}{\pi^{\frac{3}{2}} \Gamma(\frac{k-1}{2})}.
  \end{align*}
  When $k=3$, this pole is a double pole, and the Laurent series of $W_3(s)$ about $s=-2$ has principal part
  \[
    -\frac{\pi^2}{3\zeta^{(2)}(3)(s+2)^2} + \frac{24a_0\zeta^{(2)}(3)-\pi^2\gamma-\pi^2\log(4\pi)}{3\zeta^{(2)}(3)(s+2)},
  \]
  where $a_0$ is the constant term in the Laurent series for the meromorphic continuation of $\langle \mathcal{V}, E_0(\cdot,\overline{s})\rangle$ at $s=\frac{3}{2}$.

\end{theorem}

We prove this theorem in the remainder of this section. We address the
meromorphic behavior of each part of $W_k(s)$ in turn, and produce
Theorem~\ref{theorem:Wks_mero} by assembling and showing cancellation between
these parts.

\subsection{Diagonal Part}\label{sec:diagonal_part}

We recognize the diagonal part in terms of the Rankin--Selberg $L$-function
associated to $\theta^k \times \theta^k$, written $L(s,\theta^k \times
\theta^k)$ and defined by
\begin{equation*}
  L(s,\theta^k \times \theta^k)
  =
  \zeta(2s) \sum_{n \geq 1} \frac{r_k(n)^2}{n^{s + \frac{k}{2} - 1}}.
\end{equation*}
%completed zeta function.
As $y^{\frac{k}{2}}\lvert \theta^k(z) \rvert^2$ is not of rapid decay, we
interpret this $L$-function through Gupta's generalization of the Zagier
regularization method to congruence subgroups~\cite{Gupta00, ZagierRankinSelberg}.

Zagier's original argument shows how to recognize the diagonal sum as an inner product of the form $\langle \mathcal{V}, E_\infty(\cdot, \overline{s}) \rangle$.  This step does not appear explicitly in Gupta's generalization.
%
%Recognizing the diagonal sum as an inner product of the form
%$\langle \mathcal{V}, E_\infty(\cdot, \overline{s}) \rangle$ appears in Zagier's
%original argument, but not explicitly in Gupta's generalization to congruence
%subgroups.
In Corollary~\ref{cor:diagonal_from_gupta} of~\ref{sec:appendix},
we extend Gupta's argument to prove that
\begin{equation*}
  \frac{\Gamma(s + \frac{k}{2} - 1)}{(4\pi)^{s + \frac{k}{2} - 1}}
  \frac{L(s,\theta^k \times \theta^k)}{\zeta(2s)}
  =
  \langle \mathcal{V}, E_\infty(\cdot, \overline{s}) \rangle
  =
  \langle \mathcal{V}(\sigma_0 \cdot), E_0(\cdot, \overline{s}) \rangle
\end{equation*}
for $s$ in the vertical strip $1 - \frac{k}{2} < \Re(s) < \frac{k}{2}$.
We also show that this function is analytic away from $s = \frac{k}{2}, 1, 0,
1-\frac{k}{2}$, and the zeros of $\zeta(2s)$.

This function relates to the diagonal part of $W_k(s)$ by a shift of variable.
Thus the diagonal part of $W_k(s)$ has potential poles
at $s = -1, - \frac{k}{2}, -\frac{k}{2} - 1$, $-k$, and at zeros of
$\zeta(2s + k + 2)$.

For the leading pole at $s=-1$, we evaluate directly
\begin{align} \label{diagonal_rightmost_res}
  \Res_{s=-1} \sum_{m=1}^\infty \frac{r_k(m)^2}{m^{s+k}}
  =
  \lim_{X \to \infty} \frac{k-1}{X^{k - 1}} \sum_{m \leq X} r_k(m)^2
  =
  \frac{\pi^k\zeta(k-1)}{\zeta^{(2)}(k)\Gamma(\frac{k}{2})^2}.
\end{align}
The second equality is the subject of~\cite{ChoiKumchevOsburn05}, which applies a
general method for evaluating sums of positive definite quadratic forms due to
M\"uller~\cite{Muller92}.
The second pole occurs at $s = - \frac{k}{2}$ and can be understood through Corollary~\ref{cor:diagonal_from_gupta} to give the residue
\begin{align}\label{diagonal_residue_s32}
  %\Res_{s = -\frac{k}{2}} \sum_{m = 1}^\infty \frac{r_k(m)^2}{m^{s+k}} =
  \frac{(4\pi)^\frac{k}{2}}{\Gamma(\frac{k}{2})}
  \Res_{s=1} \left\langle \mathcal{V}, E_0(\cdot,\overline{s})\right\rangle.
\end{align}

The poles from zeros of the zeta function and the two remaining poles in the
diagonal part can be analyzed using the functional equation for
$L(s, \theta^k \times \theta^k)$, but these details will not be necessary as we
will show that the diagonal part identically cancels with
$\mathcal{R}_0^+ - \mathcal{R}_0^-$ in a region containing these poles.

\subsection{Discrete Part}

As discussed in \S\ref{sec:discrete_part_continuation}, the discrete part of
$W_k(s)$ is meromorphic in $\mathbb{C}$ and analytic for $\Re s> -\frac{k+3}{2}$,
where we focus our analysis.
The boundary of this region, the line $\Re s = - \frac{k+3}{2}$, hosts a line
of poles coming from the eigenvalues $t_j$ of the Maass forms.

\subsection{Continuous Part}

We now discuss the analytic properties of the continuous part of $W_k(s)$ in
the right half-plane $\Re s > - \frac{k+3}{2}$.
As shown in \S\ref{sec:continuous_part_continuation}, $W_k(s)$ has a
meromorphic continuation to the entire complex plane which incorporates many
residual terms $\mathcal{R}_{-j}^\pm$ as $\Re s$ decreases.
However, the only residual terms present in $\Re s > - \frac{k+3}{2}$ are
$\mathcal{R}_1^\pm$ and $\mathcal{R}_0^\pm$.

In analogy with~\cite{HulseKuanLowryDudaWalker17}, we expect that
$\mathcal{R}_1^+ = -\mathcal{R}_1^-$ when $w=0$.
This is correct, but is harder to prove in our current situation because
the level, 4, is not square-free.

\begin{lemma}\label{lem:R1_antisymmetry}
  With the notation of \S\ref{sec:mero}, we have
  \begin{equation*}
    \mathcal{R}_1^+(s,0) = -\mathcal{R}_1^-(s,0).
  \end{equation*}
\end{lemma}

\begin{proof}

Beginning with the formula for $\mathcal{R}_1^-$ given in~\eqref{r_1plus}, set $w=0$ and apply the Gauss duplication formula to obtain
\[
  \mathcal{R}_1^-(s,0) =-\frac{(4\pi)^\frac{k}{2}}{2} \cdot\frac{\Gamma(s-\frac{1}{2})\pi^{s-\frac{3}{2}} \left\langle \mathcal{V}, E_0(\cdot, 2-\overline{s})\right\rangle}{2\Gamma(s)\Gamma(s+\frac{k}{2}-1)}.
\]
Let $\vec{E}(z, s) = \left(E_\mathfrak{a}(z,s)\right)_\mathfrak{a}$.
Following Iwaniec~\cite{Iwaniec02}, we have $\vec{E}(z, s)= \Phi(s) \vec{E}(z, 1-s)$, in which $\Phi(s)$ is the symmetric scattering matrix
\begin{equation}\label{eq:scatter}
\Phi(s) = \pi^\frac{1}{2} \frac{\Gamma(s-\frac{1}{2})}{\Gamma(s)} \big( \varphi_{\mathfrak{ab}0}(s) \big)_{\mathfrak{a},\mathfrak{b}}
\end{equation}
composed of the constant Fourier coefficients of the various Eisenstein series $E_\mathfrak{a}(\sigma_\mathfrak{b} z,s)$.
In particular, we have that
\begin{align*}
  E_0(z,s) &= \frac{\sqrt{\pi}\Gamma(s-\frac{1}{2})}{\Gamma(s)\zeta^{(2)}(2s)}\bigg(\frac{\zeta^{(2)}(2s-1)E_\infty(z,1-s)}{4^s} \\
  &\qquad + \frac{\zeta^{(2)}(2s-1)E_\frac{1}{2}(z,1-s)}{4^s}+ \frac{\zeta(2s-1)E_0(z,1-s)}{2^{4s-1}}\bigg).
\end{align*}
We apply the Gauss duplication formula and the functional equations of $E_0(z,s)$ and the Riemann zeta function to transform $\mathcal{R}_1^-$ into
\begin{align*}
  - &\frac{(4\pi)^\frac{k}{2}\Gamma(2s-2)\pi^{2-s}\zeta(2s-2)}{2\Gamma(s)\Gamma(s+\frac{k}{2}-1)\Gamma(2-s)\zeta^{(2)}(4-2s)} \times \bigg( \frac{\langle \mathcal{V},E_0(\cdot, \overline{s}-1)\rangle}{2^{5-2s}} \\
  &\hspace{2 mm} +\frac{(4^{3-2s}-2^{3-2s})\langle \mathcal{V},E_\infty(\cdot, \overline{s}-1)\rangle}{2^{8-4s}}+\frac{(2^{3-2s}-1)\langle \mathcal{V},E_\frac{1}{2}(\cdot, \overline{s}-1)\rangle}{2^{5-2s}}\bigg).
\end{align*}
We compute the residue of $-\mathcal{R}_1^+$ given in \eqref{eq:R_1_plus_def} as we did for \eqref{r_1plus} for $\mathcal{R}_1^-$, although this time none of the cuspidal contributions vanish. Then after replacing the zeta functions with the expansions given in \eqref{eisenstein_coefficient_dirichlet_series}, term-by-term comparison shows $\mathcal{R}_1^-$ is equal to $-\mathcal{R}_1^+$.
\end{proof}

The contribution from $\mathcal{R}_1^+(s,0)-\mathcal{R}_1^{-}(s,0)$, written with arguments as they appear within the term $2Z_k(s+\frac{k}{2}+1,0)$, thus takes the form
\[
  4\mathcal{R}_1^+(s+\tfrac{k}{2}+1,0) =
 (4\pi)^\frac{k}{2}  \frac{\Gamma(s+\frac{k}{2}+\frac{1}{2}) \pi^{s+\frac{k-1}{2}} \langle \mathcal{V}, E_0(\cdot,1-\frac{k}{2}-\overline{s}) \rangle}{\Gamma(s+\frac{k}{2}+1)\Gamma(s+k)}.
 \]
This term has infinitely many poles (at least, when $k$ is even), of which at most two lie in the right half-plane $\Re s > -\frac{k+3}{2}$.  There is a pole at $s= -\frac{k}{2}$ coming from the Eisenstein series, with residue
\[
  \Res_{s = -\frac{k}{2}} 4\mathcal{R}_1^+(s+\tfrac{k}{2}+1,0) =
  -\frac{(4\pi)^\frac{k}{2} }{\Gamma(\frac{k}{2})} \Res_{s=1} \langle \mathcal{V}, E_0(\cdot, \overline{s}) \rangle.
\]

A second pole appears at $s=1-k$ from the inner product (although not from the Eisenstein series), which is relevant to our study in the cases $k \leq 4$.
In the case $k=4$, the pole at $s=1-k$ in the inner product is cancelled by a zero in $\Gamma(s+\frac{k}{2}+1)^{-1}$, and does not appear.
In the remaining case, $k=3$, this pole collides with a pole at $s=-\frac{k+1}{2}$ coming from the gamma factor, creating a double pole with principal part
\[
  -\frac{\pi^2}{3\zeta^{(2)}(3)(s+2)^2} + \frac{24a_0\zeta^{(2)}(3)-\pi^2\gamma-\pi^2\log(4\pi)}{3\zeta^{(2)}(3)(s+2)},
\]
in which $\gamma$ is the Euler-Mascheroni constant and $a_0$ is the constant coefficient of the Laurent expansion of $\langle \mathcal{V},E_0(\cdot, \overline{s})\rangle$ about $s=\frac{3}{2}$.

For $k\geq 4$, the gamma factor pole at $s=-\frac{k+1}{2}$ is simple, with residue
\[
  \Res_{s = -\frac{k+1}{2}} 4\mathcal{R}_1^+(s+\tfrac{k}{2}+1,0) =
  (4\pi)^\frac{k}{2} \frac{\langle \mathcal{V}, E_0(\cdot,\frac{3}{2}) \rangle}{\pi^\frac{3}{2}\Gamma(\frac{k-1}{2})}.
\]

Further analogy with~\cite{HulseKuanLowryDudaWalker17} leads us to expect that $\mathcal{R}_0^+(s,0) = -\mathcal{R}_0^-(s,0)$ and that $2\mathcal{R}_0^+(s,0)$ shows significant cancellation with the diagonal term.
A computation very similar to that performed in Lemma~\ref{lem:R1_antisymmetry} shows that this is indeed the case.

\begin{lemma}
  With the notation of \S\ref{sec:mero}, we have
  \begin{equation*}
    \mathcal{R}_0^+(s,0) = -\mathcal{R}_0^-(s,0).
  \end{equation*}
\end{lemma}

Simplifying $2\mathcal{R}_0^+(s,0)$ gives
\[
  2\mathcal{R}_0^+(s,0) = - \frac{1}{2} \cdot \frac{(4\pi)^{s+\frac{k}{2}-1}}{\Gamma(s+\frac{k}{2}-1)} \left\langle \mathcal{V},E_\infty(\cdot, \overline{s}) \right\rangle.
\]
As in \S\ref{sec:diagonal_part}, Zagier regularization identifies this expression with a Rankin--Selberg $L$-function,
\[
  \mathcal{R}_0^+ - \mathcal{R}_0^- =-\frac{1}{2}\frac{L(s,\theta^k \times \theta^k)}{\zeta(2s)},
\]
and we conclude that the second residual pair in the meromorphic continuation of $2Z_k(s+\frac{k}{2}+1,0)$ \textbf{\emph{exactly cancels}} with the diagonal part.  This cancellation can only occur in the half-plane $\Re s < -\frac{k+1}{2}$ and allows us to ignore $\mathcal{R}_0^\pm$ as soon as it appears.

\subsection{Non-Spectral Part}

We conclude this section with a few remarks on the polar behavior of the non-spectral part.
As it appears in $2 Z_k(s+\frac{k}{2}+1,0)$, this term takes the form
\begin{equation}\label{eq:mathfrak_E_def}
  \mathfrak{E}_k(s) = \frac{2\pi^k \Gamma(s+1)\zeta(s+1)\zeta(s+k)}{\Gamma(\frac{k}{2})\Gamma(s+\frac{k}{2}+1)\zeta^{(2)}(k)} \left(1+ \frac{1}{2^{2s+k}} - \frac{1}{2^{s+k-1}}\right).
\end{equation}
This expression is analytic in the region $\Re s > 0$ and extends meromorphically to all of $\mathbb{C}$ with poles $s=0$ and $s=-1$.  Potential poles at negative odd integers $\leq -3$ are cancelled by trivial zeta zeros, while the existence of the poles at negative even integers depends on $k$.

When $k$ is odd, $\mathfrak{E}_k(s)$ has poles at negative even integers and a double pole at $s=1-k$ coming from $\Gamma(s+1)\zeta(s+k)$.  When $k$ is even, zeros from $\zeta(s+1)\zeta(s+k)/\Gamma(s+\frac{k}{2}+1)$ cancel all but $\lfloor \frac{k}{4} \rfloor$ of these additional poles, leaving only poles at $0$, $-1$, and each negative even integer greater than $-1-\frac{k}{2}$.

We compute the residue at $s=-1$ to be
\begin{equation*}
  \Res_{s=-1} \mathfrak{E}_k(s) =  -\frac{\pi^k\zeta(k-1)}{\zeta^{(2)}(k)\Gamma(\frac{k}{2})^2},
\end{equation*}
which perfectly cancels the corresponding pole from the diagonal part in~\eqref{diagonal_rightmost_res}.
This completes the proof of Theorem~\ref{theorem:Wks_mero}. \hfill \qed

\section{Analysis of $D(s, P_k \times P_k)$}\label{sec:AnalysisD_PkPk}

We now analyze $D(s, P_k \times P_k)$.
Through the decomposition in~\eqref{eq:DsPkSquared_equals_DsSkSquared}, we relate $D(s, P_k \times P_k)$ to $D(s, S_k \times S_k)$, which further decomposes in terms of $W_k(s)$ from~\eqref{s_basic_decomp}.
Building on the analysis from the previous sections, we will show surprising amounts of cancellation in the poles and residues of $D(s, P_k \times P_k)$.

It is helpful to combine the two decompositions~\eqref{eq:DsPkSquared_equals_DsSkSquared} and~\eqref{s_basic_decomp} into the following unified formula for $D(s,P_k \times P_k)$:
\begin{align}
  &D(s, P_k \times P_k) = \zeta(s+k-2)+ W_k(s-2)+V_k^2 \zeta(s-2) \label{line1:pk_decomposition}\\
  &\qquad -  2V_k\zeta(s+\tfrac{k}{2}-2) -2V_k L(s-1,\theta^k) \label{line2:pk_decomposition} \\
  &\qquad + \frac{1}{2\pi i} \int_{(\sigma)} W_k(s-2-z) \zeta(z) \frac{\Gamma(z)\Gamma(s+k-2-z)}{\Gamma(s+k-2)}dz \label{line3:pk_decomposition} \\
  &\qquad - \frac{2V_k}{2\pi i} \int_{(\sigma)} L(s-1-z,\theta^k)\zeta(z) \frac{\Gamma(z)\Gamma(s+\frac{k}{2}-2-z)}{\Gamma(s+\frac{k}{2}-2)}dz, \label{line4:pk_decomposition}
\end{align}
initially valid with $\Re s \gg 1$ and $\sigma \in (1,\Re s -3)$.

%Note that shifting the line of integration in~\eqref{line3:pk_decomposition} to the left of $\sigma =1$ extracts a residue which can be written $W_k(s-3)/(s+k-3)$.
Since the discrete part of $W_k(s-3)$ has a line of poles where $\Re s = \frac{3-k}{2}$, we necessarily restrict our analysis of $D(s,P_k \times P_k)$ to the half-plane $\Re s > \frac{3-k}{2}$.
For ease of exposition, we further restrict ourselves to the half-plane $\Re s > 0$.

We investigate the analytic properties of $D(s, P_k \times P_k)$ by expounding each part of the decomposition given in~\eqref{line1:pk_decomposition}--\eqref{line4:pk_decomposition}.
For easy reference, a summary of the locations and residues of the poles of $D(s, P_k \times P_k)$ in the half-plane $\Re s > 0$ is provided in Table~\ref{table}.

\begin{table}[t]
\caption{Summary of Polar Data of $D(s,P_k \times P_k)$ in the Half-Plane $\Re s > 0$}\label{table}
\small
\renewcommand{\arraystretch}{1.5}
\begin{threeparttable}
\begin{tabular}{llll}
  \toprule

  \textsc{pole location} & \textsc{line} & \textsc{contributing term} & \textsc{residue} \\

  \midrule

  $s = 3$ &\eqref{line1:pk_decomposition} & $V_k^2\zeta(s-2)$ & $V_k^2$ \\
  $s = 3$ &\eqref{line3:pk_decomposition} & $\frac{\mathfrak{E}_k(s-3)}{s + k - 3}$, from $\frac{W_k(s-3)}{s + k - 3}$ & $V_k^2$ \\
  $s = 3$ &\eqref{line4:pk_decomposition} & $-2V_k \frac{L(s-2, \theta^k)}{s + \frac{k}{2} - 3}$ & $-2V_k^2$ \\

  \midrule

  $s = 2$ &\eqref{line1:pk_decomposition} & $\mathfrak{E}_k(s - 2)$, from $W_k(s-2)$ & $k V_k^2$ \\
  $s = 2$ &\eqref{line2:pk_decomposition} & $-2V_k L(s-1, \theta^k)$ & $-k V_k^2$ \\
  $s = 2$ &\eqref{line3:pk_decomposition} & $-\frac{\mathfrak{E}_k(s - 2)}{2}$, from $-\frac{W_k(s-2)}{2}$ & $-\frac{k}{2} V_k^2$ \\
  $s = 2$ &\eqref{line4:pk_decomposition} & $2 V_k \frac{L(s-1, \theta^k)}{2}$ & $\frac{k}{2}V_k^2$   \\

  \midrule

  $s = 3 - \frac{k}{2}$ &\eqref{line2:pk_decomposition} & $-2V_k \zeta(s+\frac{k}{2} - 2)$ & $-2V_k$ \\
  $s = 3 - \frac{k}{2}$ &\eqref{line4:pk_decomposition} & $-2 V_k \frac{L(s-2, \theta^k)}{s + \frac{k}{2} - 3}$ & $-2V_k L(1-\tfrac{k}{2},\theta^k)$   \\

  \midrule

  $s = 1$, if $k \neq 3$ &\eqref{line3:pk_decomposition} & $\frac{\mathfrak{E}_k(s-3)}{s + k - 3}$, from $\frac{W_k(s-3)}{s + k - 3}$ & $\frac{\pi^k \zeta(k-2)(1 + 2^{3-k})}{12\Gamma(\frac{k}{2})^2 \zeta^{(2)}(k)}$ \\
  $s = 1$ &\eqref{line3:pk_decomposition} & $\frac{\mathfrak{E}_k(s - 1)(s + k - 2)}{12}$ & $\frac{ V_k^2k(k-1)}{12}$ \\
  $s = 1$ &\eqref{line4:pk_decomposition} & $-2V_k \frac{L(s, \theta^k) (s + \frac{k}{2} - 2)}{12}$ & $-V_k \frac{\pi^{k/2}(\frac{k}{2} - 1)}{6\Gamma(k/2)}$\\
  $s = 4 - k$, if $k$ odd &\eqref{line3:pk_decomposition} & $\frac{\mathfrak{E}_k(s-3)}{s + k - 3}$, from $\frac{W_k(s-3)}{s + k - 3}$ & \text{double pole, see~\eqref{eq:doublepole_nonspectral_principalpart}} \\
  $s = 3 -\frac{k+1}{2}$, $k\neq 3$ &\eqref{line3:pk_decomposition} & $\frac{2\mathcal{R}_1^+(s+\frac{k}{2}-2,0)}{s + k - 3}$, from $\frac{W_k(s-3)}{s + k - 3}$ & $\frac{ (4\pi)^{\frac{k}{2}} \langle \mathcal{V}, E_0(\cdot, \frac{3}{2}) \rangle}{\pi^{3/2} \Gamma(\frac{k+1}{2})}$ \\
  $s = 3 -\frac{k+1}{2}$, $k=3$ &\eqref{line3:pk_decomposition} & $\frac{2\mathcal{R}_1^{+}(s+\frac{k}{2}-2,0)}{s + k - 3}$, from $\frac{W_k(s-3)}{s + k - 3}$ & \text{double pole, see~\eqref{eq:doublepole_eisen_principalpart}} \\

  %\eqref{line1:pk_decomposition} & $s = 3 - k$, if $k$ odd & $\mathfrak{E}_k(s-2)$, from $W_k(s-2)$ & *** \\
  %\eqref{line1:pk_decomposition} & $s = 2 -\frac{k+1}{2}$ & $\mathcal{R}_1^{\pm}$, from $W_k(s-2)$ & $\frac{(4\pi)^{\frac{k}{2}} V \langle \mathcal{V}, E_0(\cdot, \frac{3}{2}) \rangle}{\pi^{3/2} \Gamma(\frac{k-1}{2})}$ \\
  %\eqref{line3:pk_decomposition} & $s = 3 - k$, if $k$ odd & $-\frac{\mathfrak{E}_k(s-2)}{2}$, from $-\frac{W_k(s-2)}{2}$ & *** \\
  %\eqref{line3:pk_decomposition} & $s = 2 -\frac{k+1}{2}$ & $-\frac{\mathcal{R}_1^{\pm}}{2}$, from $-\frac{W_k(s-2)}{2}$ & $-\frac{(4\pi)^{\frac{k}{2}} V \langle \mathcal{V}, E_0(\cdot, \frac{3}{2}) \rangle}{2\pi^{3/2} \Gamma(\frac{k-1}{2})}$ \\
  %& $s = 0$, if $k \neq 3$ & $-\frac{\mathfrak{E}_k(s-2)}{2}$, from $-\frac{W_k(s-2)}{2}$ & $-\frac{\pi^k \zeta(k-2)(1 - 2^{4-k})}{2\Gamma(\frac{k}{2})\Gamma(\frac{k}{2} - 1) \zeta^{(2)}(k)}$ \\
  %& $s = 0$ & $\mathfrak{E}_k(s-2)$, from $W_k(s-2)$ & $\frac{\pi^k \zeta(k-2)(1 - 2^{4-k})}{\Gamma(\frac{k}{2})\Gamma(\frac{k}{2} - 1) \zeta^{(2)}(k)}$
  \bottomrule
\end{tabular}
\begin{tablenotes}[para,flushleft]
  \centering
  {\footnotesize
    See Proposition~\ref{prop:dirichlet} for the definition of $W_k$,~\eqref{eq:R_1_plus_def} for $\mathcal{R}_1^+$, and~\eqref{eq:mathfrak_E_def} for  $\mathfrak{E}_k$.
  }
\end{tablenotes}
\end{threeparttable}
\end{table}

\subsection*{Polar terms}
We examine the poles from terms in~\eqref{line1:pk_decomposition}
and~\eqref{line2:pk_decomposition}.
The terms occurring in the first two lines include $W_k(s-2)$ and a collection
of functions of classical interest.
The poles and residues of these terms are therefore given by
Theorem~\ref{theorem:Wks_mero} or are otherwise well-known.

\subsection*{The $W_k(s)$}

We first look at the $W_k(s)$ integral in~\eqref{line3:pk_decomposition}.
To understand the integral, we shift $\sigma$ to $-3 + \epsilon$ for some small $\epsilon > 0$ and understand the resulting residues.
There are residues at $z = 1$ from $\zeta(z)$, and at $z = 0$ and $z = -1$ from $\Gamma(z)$.
By Cauchy's Theorem, the $W_k(s)$ integral in~\eqref{line3:pk_decomposition} is equal to
\begin{align*}
  \frac{1}{2\pi i} &\int_{(-3 + \epsilon)} W_k(s - 2 - z) \zeta(z) \frac{\Gamma(z) \Gamma(s + k - 2 - z)}{\Gamma(s + k - 2)} dz \\
  &\quad + \frac{W_k(s - 3)}{s + k - 3} - \frac{W_k(s - 2)}{2} + \frac{W_k(s - 1)(s+k-2)}{12}.
\end{align*}
The integrand is now analytic for $\Re s > -1+\epsilon$, and the poles from the $z$-residues can be interpreted using Theorem~\ref{theorem:Wks_mero}.

\subsection*{The $L(s, \theta^k)$}

We now examine the $L(s, \theta^k)$ integral in~\eqref{line4:pk_decomposition}.
As with the previous integral, we shift $\sigma$ to $-3 + \epsilon$ for some small $\epsilon > 0$ and understand the resulting residues.
By Cauchy's Theorem, the $L(s, \theta^k)$ integral in~\eqref{line4:pk_decomposition} is equal to
\begin{align*}
  &\frac{-2V_k}{2\pi i} \int_{(-3 + \epsilon)} L(s - 1 - z, \theta^k) \zeta(z) \frac{\Gamma(z) \Gamma(s + \frac{k}{2} - 2 - z)}{\Gamma(s + \frac{k}{2} - 2)} dz \\
  &\quad -2V_k \bigg( \frac{L(s - 2, \theta^k)}{s + \frac{k}{2} - 3} - \frac{L(s - 1, \theta^k)}{2} + \frac{L(s, \theta^k)(s + \frac{k}{2} - 2)}{12}\bigg).
\end{align*}
The integrand is analytic for $\Re s > -1+\epsilon$.
As $L(s, \theta^k)$ is analytic except for a simple pole at $s = 1$, it is easy to recognize the poles with $\Re s> 0$ in the expression above.
Note that there is an additional pole at $s = 3 - \frac{k}{2}$ coming from the denominator of $L(s-2, \theta^k) (s + \frac{k}{2} - 3)^{-1}$.

\subsection{Examination of Poles and their Cancellation}\label{sec:cancellation}

We now begin a polar analysis of $D(s, P_k \times P_k)$ in the half-plane $\Re s > 0$.
With reference to Table~\ref{table}, we see at once that the residues of $D(s,P_k \times P_k)$ at $s=3$ and $s=2$ both vanish, hence neither of these potential poles occur.

%compute the residue of $D(s, P_k \times P_k)$ at $s=3$ as
%\begin{equation*}
%    \Res_{s=3} \bigg(V_k^2 \zeta(s-2) + \frac{W_k(s-3)}{s+k-3} - 2V_k \frac{\mathcal{L}(s-2,\theta^k)}{s+\frac{k}{2}-3}\bigg) =0.
%\end{equation*}

%The next apparent poles are at $s=2$, but these poles also conspire to cancel.
%Referring again to Table~\ref{table}, we compute
%\begin{align*}
%  \Res_{s=2} &\bigg((1-\tfrac{1}{2})\mathfrak{E}_k(s-2) - (1-\tfrac{1}{2}) 2 V_k L(s-1, \theta^k) \bigg) =0.
% \end{align*}

We now address the contribution of the poles at $s=3-\frac{k}{2}$, which are the rightmost potential poles in the $k=3$ case.
These poles occur in the terms $-2V_k \zeta(s + \frac{k}{2} - 2)$ and $L(s-2, \theta^k)(s + \frac{k}{2} - 3)^{-1}$, and combine to give the residue
\begin{equation*}
  \Res_{s=3 - \frac{k}{2}} \!\bigg(-2V_k \zeta(s + \tfrac{k}{2} - 2) + \frac{L(s-2, \theta^k)}{s + \tfrac{k}{2} - 3}\bigg) = -2V_k\big(1 + L(1 - \tfrac{k}{2},\theta^k)\big).
\end{equation*}
We evaluate $L(1 - \frac{k}{2}, \theta^k)$ using the functional equation of $L(s, \theta^k)$,
\begin{equation*}
  \pi^{-s-\frac{k}{2}+1} \Gamma(s+\tfrac{k}{2}-1) L(s,\theta^k)= \pi^{s-1}\Gamma(1-s)L(2-\tfrac{k}{2}-s,\theta^k),
\end{equation*}
and conclude that
\begin{equation*}
  L(1-\tfrac{k}{2},\theta^k) =\frac{\Gamma(\frac{k}{2})}{\pi^{\frac{k}{2}}} \lim_{s \to 0} \frac{L(1-s,\theta^k)}{\Gamma(s)} = -\frac{\Gamma(\frac{k}{2})}{\pi^{\frac{k}{2}}} \Res_{s=1} L(s,\theta^k) = -1.
\end{equation*}
Therefore, the residue at $s=3-\frac{k}{2}$ is exactly $0$, and so this pole also cancels.

There is a simple pole at $s = 3 - \frac{k+1}{2}$ in the case $k\geq 4$, with residue
\begin{align}\label{s_equals_3_minus_k2}
  \Res_{s = 3 - \frac{k+1}{2}} \frac{2\mathcal{R}_1^+(s+\frac{k}{2}-2,0)}{s+k-3} = \frac{(4\pi)^\frac{k}{2}}{\pi^\frac{3}{2} \Gamma(\frac{k+1}{2})} \left\langle \mathcal{V}, E_0(\cdot, \tfrac{3}{2}) \right\rangle.
\end{align}
When $k = 3$, this term is a double pole at $s=1$, with principal part
\begin{equation}\label{eq:doublepole_eisen_principalpart}
  -\frac{\pi^2}{3\zeta^{(2)}(3)(s-1)^2}+\frac{\pi^2(1-\gamma-\log(4\pi))}{3\zeta^{(2)}(3)(s-1)}+\frac{8a_0}{(s-1)},
\end{equation}
in which $a_0$ is the constant term in the Laurent series for the meromorphic continuation of $\langle \mathcal{V}, E_0(\cdot,\overline{s})\rangle$ at $s=\frac{3}{2}$.

In general, the poles at $s=1$ do not cancel, and constitute the leading polar term.
There are always simple poles coming from $\mathfrak{E}_k(s-1)(s+k-2)/12$ and $-2V_k L(s, \theta^k)(s+\frac{k}{2}-2)/12$, which jointly contribute the residue
\begin{equation*}
  \frac{1}{24} k^2 V_k^2.
\end{equation*}
There is also a pole at $s=1$ coming from $\mathfrak{E}_k(s-3)(s+k-3)^{-1}$, but the nature of this pole depends on $k$.
There are two cases.
If $k > 3$, there is a simple pole with residue
\begin{equation*}
 \frac{\pi^k \zeta(k-2)}{12\,\Gamma(\frac{k}{2})^2\zeta^{(2)}(k)}\left(1+2^{3-k}\right).
\end{equation*}
If $k = 3$, then there is a double pole with principal part
\begin{equation}\label{eq:doublepole_nonspectral_principalpart}
  \frac{2\pi^2}{3\zeta^{(2)}(3)(s-1)^2}+\frac{\pi^2(2\gamma + \log 2 -24\zeta'(-1))}{3\zeta^{(2)}(3)(s-1)},
\end{equation}
Altogether, the analysis of \S\ref{sec:cancellation} leads to the following theorem.

\begin{theorem}\label{theorem:DsPkPk}
  The Dirichlet series $D(s,P_k \times P_k)$, defined originally in the right half-plane $\Re s > 3$ by the series
  \[
    \sum_{m=1}^\infty \frac{P_k(m)^2}{m^{s+k-2}},
  \]
  has a meromorphic continuation to $\mathbb{C}$ given by~\eqref{line1:pk_decomposition}--~\eqref{line4:pk_decomposition} and is analytic in the right half-plane $\Re s > 1$, with a pole at $s=1$.
  In the case $k \geq 4$ this pole is simple, with residue
  \[
    \frac{k^2}{24} V_k^2 + \frac{\pi^k \zeta(k-2)}{12\,\Gamma(\frac{k}{2})^2\zeta^{(2)}(k)}\left(1+2^{3-k}\right).
  \]
  In the case $k=3$ this is a double pole, with principal part given by
 \begin{equation*}
   \frac{\pi^2}{3\zeta^{(2)}(3)(s-1)^2}+\frac{\pi^2\left(1+\gamma - \log (2\pi) -24\zeta'(-1)+2\zeta^{(2)}(3)\right)+24a_0\zeta^{(2)}(3)}{3\zeta^{(2)}(3)(s-1)}.
\end{equation*}
The function $D(s,P_k \times P_k)$ is otherwise analytic in the right half-plane $\Re s > \frac{3-k}{2}$ save for finitely many poles at non-positive integers and, for $k>3$, an additional simple pole at $s=\frac{5-k}{2}$ with residue given by~\eqref{s_equals_3_minus_k2}.
\end{theorem}

\begin{remark}
  In the process of proving this theorem, we have also shown that $D(s, S_k \times S_k)$ has a meromorphic continuation to $\mathbb{C}$.
  The poles and residues of $D(s, S_k \times S_k)$ can be recovered from the analysis of $D(s, P_k \times P_k)$ and the decomposition~\eqref{eq:DsPkSquared_equals_DsSkSquared}.
\end{remark}

\begin{remark}
  The simple pole at $s = \frac{5-k}{2}$ is particularly interesting in the case $k=4$, when it appears in the right half-plane $\Re s > 0$.  In this case, noting that $r_4(m)/8$ is multiplicative and comparing Euler products shows that
\[
  \frac{1}{64}\sum_{m=1}^\infty \frac{r_4(m)^2}{m^s}
  =
  \frac{(2^{6-3s} - 5 \cdot 2^{3-2s} + 2^{1-s}+1)\zeta(s-2)\zeta^2(s-1)\zeta(s)}{(1+2^{1-s})\zeta(2s-2)},
\]
  which can be used to evaluate the inner product $\left\langle \mathcal{V}, E_0(\cdot, \frac{3}{2})\right\rangle$ appearing in~\eqref{s_equals_3_minus_k2} via~\eqref{diagonal_residue_s32}.
  The residue of $D(s, P_4 \times P_4)$ at $s=\frac{1}{2}$ is given by
  \[
    C_4'
    :=
    \frac{16(9\sqrt{2}-8)\zeta(\frac{1}{2})
    \zeta(\frac{3}{2})^2\zeta(\frac{5}{2})}{7 \pi^2 \zeta(3)}.
  \]
\end{remark}

\section{Smooth Second Moment}\label{sec:secondmoment}

In this section, we use the meromorphic properties of $D(s, P_k \times P_k)$ to prove our main smoothed result regarding estimates for $\sum P_k(n)^2 e^{-n/X}$.
Key to this approach is the exponential cutoff transform
\begin{equation}\label{eq:smooth_cutoff_DsPkPk}
  \frac{1}{2\pi i} \int_{(4)} D(s, P_k \times P_k) X^{s+k-2} \Gamma(s+k-2) \, ds = \sum_{n \geq 1} P_k(n)^2 e^{-n/X}.
\end{equation}

We may evaluate the left-hand side of the inverse Mellin transform in~\eqref{eq:smooth_cutoff_DsPkPk} by decomposing $D(s,P_k \times P_k)$ as in~\eqref{line1:pk_decomposition}--\eqref{line4:pk_decomposition} and then shifting the lines of integration from $\Re s = 4$ to $\Re s = \epsilon$.
From Theorem~\ref{theorem:DsPkPk}, we understand that these integration shifts pass by a pole at $s=1$ (which is a double pole for $k=3$) and a pole at $s =\frac{1}{2}$ (if $k=4$).

Provided that the integral in~\eqref{eq:smooth_cutoff_DsPkPk} converges away from poles on each abscissa $(\sigma)$ for $\sigma \in (0,4)$, we would have
\begin{align}
    &\sum_{n \geq 1} P_k(n)^2 e^{-n/X} = \delta_{[k=3]} C'_3  X^{k-1} \left(\log X +1-\gamma\right)+ C_k \Gamma(k-1)X^{k-1} \notag \\
    &+ \delta_{[k = 4]} \Gamma(\tfrac{5}{2}) C'_4 X^{k-\frac{3}{2}} + \frac{1}{2\pi i}\! \int_{(\epsilon)}\!\! D(s, P_k \times P_k) X^{s+k-2} \Gamma(s+k-2) ds. \label{eq:smooth_integral_shift}
\end{align}
Here, the constants $C_k$, $C_3'$, and $C_4'$ are given explicitly by the Laurent coefficients of $D(s,P_k \times P_k)$ about its singular points, as described in Remark~\ref{rem:smooth_coefficients} and Theorem~\ref{theorem:DsPkPk}.

Since $\Gamma(s)$ experiences exponential decay as $\lvert \Im s \rvert \to \infty$, it suffices to show that $D(s,P_k \times P_k)$ grows at most polynomially in $\lvert \Im s \rvert$.
We will accomplish this through a series of lemmas.

\begin{lemma}\label{lem:Wks_polynomial_growth}
  The function $W_k(s)$ is bounded polynomially in $\lvert \Im s \rvert$ away from poles in  vertical strips.
\end{lemma}

\begin{proof}
  We prove this by showing that the diagonal, non-spectral, discrete, and continuous parts of $W_k(s)$ grow at most polynomially in $\lvert \Im s \rvert$.

For the diagonal part this is a consequence of the Phragm\'en-Lindel\"of principle and the existence of a functional equation to give bounds for $L(s,\theta^k \times \theta^k)$ in a left half-plane.
(See \S\ref{sec:diagonal_part}.)

For the non-spectral part $\mathfrak{E}_k(s)$, we obtain at most polynomial growth in $\lvert \Im s \rvert$ as a consequence of polynomial bounds on $\zeta(s)$ and Stirling's approximation for the gamma ratio $\Gamma(s+1)/\Gamma(s+\frac{k}{2}+1)$.

In the continuous part, we must address the growth of $\mathcal{R}_{-j}^\pm$ as well as the integral~\eqref{continuous_part_simplified}.
To bound
\[
  2\mathcal{R}_1^+(s+\tfrac{k}{2}+1,0) =
 (4\pi)^\frac{k}{2}  \frac{\Gamma(s+\frac{k}{2}+\frac{1}{2}) \pi^{s+\frac{k-1}{2}} \langle \mathcal{V}, E_0(\cdot,1-\frac{k}{2}-\overline{s}) \rangle}{\Gamma(s+\frac{k}{2}+1)\Gamma(s+k)},
 \]
we recall that $\langle \mathcal{V}, E_0(\cdot,1-\frac{k}{2}-\overline{s}) \rangle$ may be identified with an $L$-function through Corollary~\ref{cor:diagonal_from_gupta} and therefore grows like a gamma function multiplied by an $L$-function of polynomial growth.
Via Stirling's approximation we see that the exponential contributions within $\mathcal{R}_1^\pm$ cancel, so $\mathcal{R}_1^\pm$ grows at most polynomially in $\lvert \Im s \rvert$.
Further terms $\mathcal{R}_{-j}^\pm$ may be treated in the same way.

To complete our analysis of the continuous part of $W_k(s)$ we need only estimate~\eqref{continuous_part_simplified} in various vertical strips. To do so, we note that $\langle \mathcal{V}, E_\mathfrak{a}(\cdot, \frac{1}{2}-\overline{z})\rangle/\Gamma(\frac{1}{2}+z)$ and $\zeta_\mathfrak{a}(s,z)$ experience at most polynomial growth in $\lvert \Im z \rvert$ and $\lvert \Im s \rvert$, and that Stirling's approximation gives
\begin{align*}
G(s,z) &=\frac{\Gamma(s-\frac{1}{2}+z)\Gamma(s-\frac{1}{2}-z)}{\Gamma(s+\frac{k}{2}-1)\Gamma(s)}  \\
&\ll \lvert \Im (s-z) \rvert^{\Re s} \lvert \Im (s+z) \rvert^{\Re s} \lvert \Im s \rvert^{A} e^{-\pi \max(\lvert \Im s \rvert, \lvert \Im z \rvert)+\pi \lvert \Im s \rvert}
\end{align*}
when $\Re z = 0$, for some constant $A$.

In the $z$-interval of length $2 \lvert \Im s \rvert^{1+\epsilon}$ where $\lvert \Im z \rvert < \lvert \Im s \rvert^{1+\epsilon}$, the exponential factors cancel and the integrand experiences polynomial growth in $\lvert \Im s\rvert$.
If $\lvert \Im z \rvert > \lvert \Im s \rvert^{1+\epsilon}$, the integrand decays exponentially.
In total, the integral contributes only polynomial growth.

Finally, we address the discrete part of $W_k(s)$.
For this,~\cite[Proposition 13]{Kiral15} shows that the inner products $\langle \mathcal{V},\mu_j \rangle$ decay exponentially in $\lvert t_j \rvert$; namely,
\[\sum_{T \leq \lvert t_j \rvert \leq  2T} \lvert \left\langle\mathcal{V},\mu_j \right\rangle \rvert^2 \ll T^{4k+2}e^{-\pi T}.\]
This exponential decay is balanced by exponential growth within the Fourier coefficients $\rho_j(h)$.
We have the estimate
\[\sum_{T \leq \lvert t_j \rvert \leq 2T} \lvert \rho_j(h) \rvert^2 e^{-\pi t_j} \ll h^{2\eta}T^2\]
given in~\cite[(4.3)]{HoffsteinHulse13}, where $\eta$ is the best-known progress toward the (non-archimedean) Ramanujan conjecture. Using this with the Cauchy-Schwarz inequality we get that
\[
\sum_{T \leq \lvert t_j \rvert \leq 2T} \lvert \rho_j(h) \rho_j(m) \rvert e^{-\pi t_j} \ll (hm)^{\eta}T^2.
\] So for $\Re(s)>2$, we have that
\begin{equation} \label{eq:lbound1}
\sum_{T \leq \lvert t_j \rvert \leq 2T} \lvert L(s,\mu_j) \rvert^2 e^{-\pi t_j} \ll T^2,
\end{equation}
and from Stirling's approximation and the functional equation, we similarly get that, when restricted to vertical strips $A<\Re(s)<-1$,
\begin{equation} \label{eq:lbound2}
\sum_{T \leq \lvert t_j \rvert \leq 2T} \lvert L(s,\mu_j) \rvert^2 e^{-\pi t_j} \ll_A T^{4-2\sigma}(1+|\Im(s)|)^{2-2\sigma}.
\end{equation}
Since each $L(s,\mu_j)$ is entire, from the Phragm\'en-Lindel\"{o}f convexity argument we have that $\sum_{T \leq \lvert t_j \rvert \leq 2T} \lvert L(s,\mu_j) \rvert^2 e^{-\pi t_j}$ has at most polynomial growth in $T$ and $\Im(s)$ when $s$ is confined to any vertical strip in $\mathbb{C}$.
%\[
%  \sum_{t_j \leq \lvert \Im s \rvert} \lvert \rho_j(1) \left \langle \mathcal{V}, \mu_j \right\rangle \rvert \ll \lvert \Im s \rvert^{k + \frac{5}{2} + \epsilon}.
%\]
Using this along with our previous bound on $G(s,z)$,
%, and the observation that $L(s,\mu_j)$ experiences polynomial growth in vertical strips,
we bound the discrete part of $W_k(s)$ polynomially in $\lvert \Im s \rvert$ via partial summation.
\end{proof}

A second lemma will be used to bound the growth of the two Mellin-Barnes integrals~\eqref{line3:pk_decomposition} and~\eqref{line4:pk_decomposition} that appear in the meromorphic continuation of $D(s, P_k \times P_k)$.

\begin{lemma}\label{lem:MBtransform_polynomial_growth}
  Let $F(s)$ be a function of polynomial growth in $\lvert \Im s \rvert$ on fixed vertical lines and let $c$ be fixed. There exists $M>0$ such that
  \[
    \frac{1}{2\pi i} \int_{(\sigma)} F(s-z) \zeta(z) \frac{\Gamma(z) \Gamma(s+c-z)}{\Gamma(s+c)} \, dz \ll \lvert \Im s \rvert^M,
  \]
in which $\sigma$ is chosen to avoid poles in the integrand and the implicit constant does not depend on $\lvert \Im s \rvert$.
\end{lemma}

\begin{proof}
  By Stirling's approximation and polynomial growth in vertical strips for both $F(s-z)$ and $\zeta(z)$, we bound our integrand by
  \[
    \lvert \Im(s-z) \rvert^{A} \lvert \Im z \rvert^{B} \lvert \Im s \rvert^{C }e^{-\frac{\pi}{2} \lvert \Im z \rvert - \frac{\pi}{2} \lvert \Im(s-z) \rvert + \frac{\pi}{2} \lvert \Im s\rvert}
  \]
  for some $A,B,C$ independent of $\lvert \Im s \rvert$ and $\lvert \Im z \rvert$.
  Growth and decay of the integrand depends on the relative sizes of $\Im s$, $\Im z$, and $\Im(s-z)$.  By casework we conclude that the integrand has exponential decay in $\lvert \Im z \rvert$ everywhere except when $\lvert \Im z \rvert \leq \lvert \Im s \rvert$, in which case the exponentials cancel.
  Thus the integrand is polynomially bounded and effectively supported on an interval of length $2\lvert \Im s \rvert^{1 + \epsilon}$, leading to a polynomial bound in $\lvert \Im s\rvert$ overall.
\end{proof}

Combining our lemmas, we bound $D(s, P_k \times P_k)$ in vertical strips and prove the following theorem.
 \begin{theorem}\label{thm:smooth_cutoff}
  For $k \geq 3$ and any $\epsilon > 0$,
  \begin{align*}
      \sum_{n = 1}^\infty P_k(n)^2 e^{-n/X} &= \delta_{[k=3]} C'_3  X^{k-1} \left(\log X +1-\gamma\right)+ C_k \Gamma(k-1)X^{k-1}  \\
      &\quad + \delta_{[k = 4]} C'_4 \Gamma(k-\tfrac{3}{2})X^{k - \frac{3}{2}} + O_\epsilon(X^{k-2 + \epsilon}),
  \end{align*}
  where $C_k$, $C_3'$, and $C_4'$ are the explicit constants described in Remark~\ref{rem:smooth_coefficients}.
\end{theorem}

\begin{proof}
  As described at the start of this section, it suffices to shift the line of integration as in~\eqref{eq:smooth_integral_shift}.
To justify this contour shift, we bound $D(s,P_k \times P_k)$ polynomially in $\lvert \Im s \rvert$ in vertical strips.
  We do so by showing a contribution of at most polynomial growth for each term in~\eqref{line1:pk_decomposition}--\eqref{line4:pk_decomposition}.

In~\eqref{line1:pk_decomposition} these bounds follow from Lemma~\ref{lem:Wks_polynomial_growth} and polynomial estimates for the Riemann zeta function in vertical strips.
For~\eqref{line2:pk_decomposition} we require a polynomial bound on $L(s,\theta^k)$ in vertical strips as well, which follows from the functional equation of $L(s,\theta^k)$ and the Phragm\'en-Lindel\"of Principle.
Finally, since $W_k(s)$ and $L(s,\theta^k)$ experience polynomial growth in vertical strips, Lemma~\ref{lem:MBtransform_polynomial_growth} gives a polynomial bound in $\lvert \Im s \rvert$ in~\eqref{line3:pk_decomposition} and~\eqref{line4:pk_decomposition}.
\end{proof}

\begin{remark} The leading constants $C_3'$ and $C_k$ ($k \geq 4$) are described explicitly in Remark~\ref{rem:smooth_coefficients}.
  In particular, we may verify that they are positive.

For small $k>3$ it is not difficult to list the precise locations of the poles of $D(s,P_k \times P_k)$ in the right half-plane $\Re s > \frac{3-k}{2}$ and derive additional main terms and improved error estimates in Theorem~\ref{thm:smooth_cutoff}.  For example, there exist constants $D_4$ and $D_5$ for which
\begin{align*}
\sum_{n \geq 1} P_4(n)^2e^{-n/X} &= 2C_4 X^3 + C_4'\Gamma(\tfrac{5}{2})X^\frac{5}{2}  + D_4X^2 +O\left(X^{\frac{3}{2}+\epsilon}\right),\\
\sum_{n \geq 1} P_5(n)^2 e^{-n/X} &= 6C_5 X^4 + D_5X^3 + O \left(X^{2+\epsilon}\right).
\end{align*}
The existence of infinitely many poles for $D(s,P_k \times P_k)$ on the line $\Re s = \frac{3-k}{2}$ suggests that these are essentially the best smooth results possible.
\end{remark}

\section{Sharp Second Moment}\label{sec:sharp_second_momentv2}

We now prove a second moment result without smoothing.
The key observation is that by Lemmas~\ref{lem:Wks_polynomial_growth}
and~\ref{lem:MBtransform_polynomial_growth}, the Dirichlet series
$D(s, P_k \times P_k)$ has polynomial growth in vertical strips (away from
poles).
Using this polynomial growth, applying Perron's formula yields a sharp moment.
In this section, we prove the following theorem.

\begin{theorem}\label{thm:sharp_cutoff_theoremv2}
  For each $k \geq 3$ there exists a $\lambda>0$ such that
  \begin{equation*}
    \sum_{n \leq X} P_k(n)^2
    =
    \delta_{[k=3]}X^{k-1} \left(\frac{C_3'}{2}\log X -\frac{C_3'}{4}\right)
    +
    \frac{C_k}{k-1} X^{k-1}
    +
    O_\lambda (X^{k-1-\lambda} ).
  \end{equation*}
  The constants $C_3'$ and $C_k$ are the same constants as in
  Remark~\ref{rem:smooth_coefficients}.
\end{theorem}

Applying the statement of Perron's formula from Theorem~5.2 and Corollary~5.3
of~\cite{montgomery2006multiplicative} with $a_n = P_k(n)^2$ and
$\sigma_0 = k - 1 + \frac{1}{\log X}$, we find that
\begin{equation}\label{eq:perron_base}
  \sum_{n \leq X} P_k(n)^2
  =
  \frac{1}{2 \pi i} \int_{\sigma_0 - iT}^{\sigma_0 + iT}
  D(s - k + 2, P_k \times P_k) \frac{X^s}{s} ds
  +
  \mathcal R,
\end{equation}
where the remainder term is bounded by
\begin{equation}\label{eq:perron_remainder}
  \mathcal{R}
  \ll
  \sum_{\substack{X/2 < n < 2X \\ n \neq X}} P_k(n)^2
  \min \Big(1, \frac{X}{T\lvert X - n \rvert} \Big)
  +
  \frac{X^{\sigma_0}}{T}
  \sum_{n \geq 1} \frac{P_k(n)^2}{n^{\sigma_0}}.
%  \\
%  &\ll
%  \sum_{\substack{X/2 < n < 2X \\ n \neq X}} P_k(n)^2
%  \min \Big(1, \frac{X^{1 - \delta}}{X - n} \Big)
%  +
%  X^{k - 1 - \delta + \epsilon}
\end{equation}

Shifting the line of integration in~\eqref{eq:perron_base} to
$k - 1 - \tfrac{1}{4}$ passes a pole at $s = k - 1$, and the residue gives the
main term in the Theorem.
There exists an $M$ such that $D(s,P_k \times P_k) s^{-1} \ll \lvert \Im s \rvert^M$
when $\Re s \geq (k-1-\frac{1}{4})$, and thus letting $T = X^\delta$ for a small
$\delta > 0$, the shifted integral (as well as the integrals
along the top and bottom of the rectangular contour) is bounded
by $O(X^{k - 5/4 + M \delta} + X^{k - 1 - \delta})$.

Now consider the remainder term $\mathcal{R}$.
The last term in the bound of $\mathcal{R}$ is itself bounded by $O_\epsilon(X^{k-1 - \delta +
\epsilon})$ for any $\epsilon > 0$.
For $k > 3$, the bound $P_k(n)^2 \ll n^{k-2} \log^{4/3} n$
(see \S2 of~\cite{IvicSurvey2006} for a survey of these results) is enough
to bound the first term by
$O_\epsilon(X^{k - 1 - \delta + \epsilon})$ for any $\epsilon > 0$.

When $k = 3$, individual bounds for $P_k(n)^2$ are too weak, but we can use
the following short interval estimate.

\begin{lemma}\label{lem:concentrating_cutoff_boundv2}
  There exists $M > 0$ such that
  \begin{equation*}
    \frac{1}{2\pi i}\int_{(k)} \!D(s-k+2,P_k \times P_k) X^s
    \exp\left( \frac{\pi s^2}{y^2}\right) \frac{ds}{y}
    \ll
    \frac{X^{k-1} \log X}{y} + X^{k-\frac{5}{4}} y^M.
  \end{equation*}
  Correspondingly, there exists $0 < \beta < 1$ such that
  \begin{equation*}
    \sum_{\lvert n - X \rvert \ll X^\beta} P_k(n)^2
    \ll
    X^{k - 2 + \beta}\log X.
  \end{equation*}
\end{lemma}

\begin{proof}
  Shift the contour left to $(k-1-\frac{1}{4})$.
  This passes a pole at $s=k-1$ with residue bounded by $O\big(X^{k-1} (\log X)/y\big)$.
  Recalling that $D(s,P_k \times P_k) \ll \lvert \Im s \rvert^M$
  when $\Re s \geq (k-1-\frac{1}{4})$, the shifted integral is bounded by
  \begin{equation*}
    \frac{X^{k-1-\frac{1}{4}}}{y}
    \int_{-\infty}^{\infty}
    (1+ \lvert t \rvert)^M \exp\left(-\frac{\pi t^2}{y^2}\right) dt
    \ll
    X^{k-1-\frac{1}{4}}y^M.
  \end{equation*}
  %This proves the first bound.
  For the second statement in the lemma, let
  $V_{X, y}(n) = (2\pi)^{-1} \exp\left( -\frac{y^2 \log^2(X/n)}{4\pi}\right)$
  denote the inverse Mellin transform of $\exp (\pi s^2 / y^2)$.
  Then
  \begin{equation*}
    \sum_{\lvert X - n \rvert < X/y} P_k(n)^2
    \ll
    \sum_{\lvert X - n \rvert < X/y} P_k(n)^2 V_{X, y}(n)
    \ll
    \sum_{n \geq 1} P_k(n)^2 V_{X, y}(n),
  \end{equation*}
  and this last sum is exactly equal to the integral in the statement of the
  lemma.
  Choosing $y = X^{1/(4(M + 1)}$ in the integral bound proves the
  short interval result on intervals of length
  $X^\beta$ with $\beta = 1 - 1/4(M+1)$.
\end{proof}

Let $\beta$ be as in the lemma, and split the first sum
in~\eqref{eq:perron_remainder} over the intervals $[X/2, X - X^\beta]$,
$[X-X^\beta, X+X^\beta]$, and $[X+X^\beta, 2X]$.
On the middle interval, Lemma~\ref{lem:concentrating_cutoff_boundv2} direcly
gives the bound $O(X^{k-2+\beta}\log X)$.
On the first and last intervals, Abel summation and the lemma imply the
bound $O(X^{k - 1 - \delta}\log^2 X)$.
Choosing $\delta$ such that $\delta < 1 - \beta$, and $\epsilon$ such
that $\epsilon < (1 - \beta) / 2$ proves the theorem with $\lambda = (1 -
\beta)/2$.

\section{Laplace Transform}\label{sec:laplace}

Theorem~\ref{thm:smooth_cutoff} may be considered as a discrete Laplace transform of the mean square of the lattice point discrepancy.
Building upon this result, one can obtain asymptotics for the continuous Laplace transform
\begin{equation}\label{eq:continuous_laplace_def}
  \int_0^\infty P_k(t)^2 e^{-t/X} dt.
\end{equation}
In this section, we prove the following estimate for the continuous Laplace transform of $P_k(t)^2$.

\begin{theorem}\label{thm:continuous_laplace}
The Laplace transform of the second moment of the lattice point discrepancy in dimensions $k \geq 3$ satisfies
\begin{align*}
    \int_0^\infty P_k(t)^2 e^{-t/X}dt &= \delta_{[k=3]} C_3' X^{k-1}(\log X + 1 -\gamma) + \delta_{[k=4]} C_4' \Gamma(k-\tfrac{3}{2}) X^{k-\frac{3}{2}} \\
     &\quad +C_k \Gamma(k-1)X^{k-1} - \frac{\Gamma(k-1)\pi^k}{6\Gamma(\frac{k}{2})^2} X^{k-1} + O\left(X^{k-2+\epsilon}\right)\!,
  \end{align*}
  where the constants are the same constants as in Remark~\ref{rem:smooth_coefficients}.
\end{theorem}

\begin{remark}
  It is possible to adapt the method of the proof of Theorem~\ref{thm:continuous_laplace} to obtain further secondary terms and decrease the error to $O(X^{\frac{k-1}{2}+\epsilon})$.
\end{remark}

Our proof of Theorem~\ref{thm:continuous_laplace} begins with the identity
\begin{equation*}
  P_k(t) = S_k(t) - V_k t^{\frac{k}{2}} = S_k( \lfloor t \rfloor ) - V_k t^{\frac{k}{2}} = P_k(\lfloor t \rfloor) + V_k \lfloor t \rfloor ^{\frac{k}{2}} - V_k t^{\frac{k}{2}}.
\end{equation*}
It follows that
\begin{equation}\label{eq:Pk_decomp_with_floors}
  P_k(t)^2 = P_k(\lfloor t \rfloor)^2 + V_k^2 \big( \lfloor t \rfloor^{\frac{k}{2}} - t^{\frac{k}{2}} \big)^2 + 2V_k P_k(\lfloor t \rfloor) \big( \lfloor t \rfloor^{\frac{k}{2}} - t^{\frac{k}{2}} \big).
\end{equation}

We will compute the Laplace transform~\eqref{eq:continuous_laplace_def} by computing it separately for each term in~\eqref{eq:Pk_decomp_with_floors}.
We begin with the first term in~\eqref{eq:Pk_decomp_with_floors}, which is very nearly equivalent to the sum studied in Theorem~\ref{thm:smooth_cutoff}.

\begin{lemma}[First term in the Laplace transform of~\eqref{eq:Pk_decomp_with_floors}]\label{lem:laplace_I} We have
  \begin{align*}
    \int_0^\infty P_k(\lfloor t \rfloor)^2 e^{-t/X} dt &= \delta_{[k=3]} C'_3  X^{k-1} \left(\log X +1-\gamma\right)+ C_k \Gamma(k-1)X^{k-1}  \\
    &\quad + \delta_{[k = 4]} C'_4 \Gamma(k-\tfrac{3}{2})X^{k - \frac{3}{2}} + O_\epsilon(X^{k-2 + \epsilon}),
  \end{align*}
\end{lemma}

\begin{proof}
  We note that
  \[ e^{-1/X} \sum_{n \geq 0} P_k(n)^2 e^{-n/X} \leq \int_0^\infty P_k(\lfloor t \rfloor)^2 e^{-t/X} dt \leq \sum_{n \geq 0} P_k(n)^2 e^{-n/X}. \]
As $e^{-1/X} = 1 + O(\frac{1}{X})$, the lemma follows from Theorem~\ref{thm:smooth_cutoff}.
\end{proof}

%\begin{proof}
%We directly compute
%\begin{align*}
%  \int_0^\infty P_k(\lfloor t \rfloor)^2 e^{-t/X} dt &= \sum_{n \geq 0} P_k(n)^2 \int_n^{n+1} e^{-t/X} dt \\
%  &= X(1 - e^{-1/X}) \sum_{n \geq 0} P_k(n)^2 e^{-n/X}.
%\end{align*}
%The sum is essentially the object of study in Theorem~\ref{thm:smooth_cutoff}.
%Noting that $X(1 - e^{-1/X}) = 1 + O(1/X)$ and simplifying completes the proof.
%\end{proof}

The second term in~\eqref{eq:Pk_decomp_with_floors} can be understood through Abel summation.
\begin{lemma}[Second term in the Laplace transform of~\eqref{eq:Pk_decomp_with_floors}]\label{lem:laplace_II} We have
  \begin{equation*}
    V_k^2 \int_0^\infty \big( \lfloor t \rfloor^{\frac{k}{2}} - t^{\frac{k}{2}} \big)^2 e^{-t/X} dt = \frac{k^2 V_k^2 \Gamma(k-1)}{12}X^{k-1} + O(X^{k-2+\epsilon}).
  \end{equation*}
\end{lemma}
\begin{proof}
  Expanding the integral and estimating the integrand, we compute
  \begin{align}
    &\int_0^\infty \big( \lfloor t \rfloor^{\frac{k}{2}} - t^{\frac{k}{2}} \big)^2 e^{-t/X} dt \notag \\
   % = \sum_{n \geq 0} \int_0^1 \left( (n+t)^{\frac{k}{2}} - n^{\frac{k}{2}} \right)^2 e^{-(n+t)/X} \,dt  \notag \\
    &\quad = \sum_{n \geq 1} n^k e^{-n/X} \int_0^1 \bigg( \!\left( 1+\frac{t}{n} \right)^{\frac{k}{2}} - 1\bigg)^2 e^{-t/X} \,dt +O(1) \notag \\
    &\quad = \frac{k^2}{4} \sum_{n \geq 1} n^k e^{-n/X} \int_0^1 \left( \frac{t}{n} + O\left( \frac{t^2}{n^2} \right) \!\right)^2 \left(1 + O \left(\frac{1}{X} \right)\! \right) \,dt +O(1) \notag \\
    &\quad = \frac{k^2}{4} \sum_{n \geq 1} n^{k-2} e^{-n/X} \int_0^1 t^2 \left(1+ O \left(\frac{1}{n} + \frac{1}{X} \right)\! \right) \,dt + O(1) \notag \\
    &\quad = \frac{k^2}{4} \sum_{n \geq 1} n^{k-2} e^{-n/X} \left(\frac{1}{3} + O\left(\frac{1}{n} + \frac{1}{X} \right) \!\right) +O(1). \label{lemma_84_sum}
  \end{align}
  The $O(1)$ term above comes from the part of the integral corresponding to $[0,1]$. For $\sigma > m+1$ with $m > 0$, we note that
  \[ \sum_{n \geq 1} n^m e^{-n/X} = \frac{1}{2\pi i} \int_{(\sigma)} \zeta(s-m)X^s\Gamma(s) \,ds = \Gamma(m+1) X^{m+1} + O(X). \]
The last equality is obtained by moving the line of integration to $\sigma = 1$, picking up the residue at $s = m+1$ and bounding the leftover integral. Combining this statement with \eqref{lemma_84_sum} gives us the lemma.
\end{proof}

Finally, we address the last term in~\eqref{eq:Pk_decomp_with_floors}.

\begin{lemma}[Third term in the Laplace transform of~\eqref{eq:Pk_decomp_with_floors}]\label{lem:laplace_III} We have
  \begin{equation*}
  2V_k \int_0^\infty \!P_k(\lfloor t \rfloor) (\lfloor t \rfloor^{\frac{k}{2}} - t^{\frac{k}{2}}) e^{-t/X} dt= -\frac{\pi^{\frac{k}{2}} k \Gamma(k - 1) V_k}{4 \Gamma(\frac{k}{2})} X^{k-1} + O(X^{k-2+\epsilon}).
  \end{equation*}
\end{lemma}

\begin{proof}
Our approach here is analogous to that of the previous lemma. We compute
\begin{align}
  &\int_0^\infty P_k(\lfloor t \rfloor) (\lfloor t \rfloor^{\frac{k}{2}} - t^{\frac{k}{2}}) e^{-t/X} \,dt \notag \\
  &\quad = -\sum_{n \geq 1} P_k(n) n^\frac{k}{2} e^{-n/X} \int_0^1 \left( \left(1+\frac{t}{n} \right)^\frac{k}{2} - 1 \right) e^{-t/X} \,dt + O(1) \notag \\
  &\quad = -\frac{k}{2} \sum_{n \geq 1} P_k(n) n^\frac{k}{2} e^{-n/X} \int_0^1 \left( \frac{t}{n} + O \left(\frac{t^2}{n^2} \right) \right) \left( 1 + O\left(\frac{1}{X} \right)\right) \,dt + O(1) \notag \\
  &\quad =  -\frac{k}{2} \sum_{n \geq 1} P_k(n) n^{\frac{k}{2}-1} e^{-n/X} \int_0^1 t \left( 1 + O \left( \frac{1}{n}+\frac{1}{X} \right)\right) \,dt + O(1) \notag \\
  &\quad = -\frac{k}{2} \sum_{n \geq 1} P_k(n) n^{\frac{k}{2}-1} e^{-n/X} \left( \frac{1}{2} + O \left(\frac{1}{n} + \frac{1}{X} \right) \right) + O(1). \label{eq:Pkexpr}
\end{align}
At this point, we transform the sum above into an inverse Mellin transform,
\begin{equation}
 \sum_{n=1}^\infty P_k(n) n^{\frac{k}{2}-1} e^{-n/X} = \frac{1}{2\pi i} \int_{(\sigma)} D(s-\tfrac{k}{2}+1,P_k) \Gamma(s) X^s ds, \label{eq:laplace_j1}
\end{equation}
in which $D(s, P_k) := \sum_{n \geq 1} P_k(n) n^{-s}$ denotes the non-normalized Dirichlet series associated to $P_k$.
By modifying the analysis of $\sum S_k(n)n^{-(s + \frac{k}{2} - 2)}$ from~\eqref{eq:Skseries} and recalling that $P_k(n) = S_k(n) - V_k n^{\frac{k}{2}}$, we see that
\begin{align}
\begin{split} \label{eq:DsPk_continuation}
  D(s,P_k) &= \zeta(s) + L(s-\tfrac{k}{2}+1,\theta^k) -V_k \zeta(s-\tfrac{k}{2}) \\
  &\quad + \frac{1}{2\pi i} \int_{(\sigma)} L(s-\tfrac{k}{2}+1-z, \theta^k)\zeta(z) \frac{\Gamma(z) \Gamma(s-z)}{\Gamma(s)} dz,
\end{split}
\end{align}
in which $L(s,\theta^k)$ is defined as in Proposition~\ref{prop:DsPkSquared_vs_DsSkSquared}.

The function $D(s, P_k)$ admits potential poles at $s = \tfrac{k}{2} + 1$ (coming from a zeta function and the Mellin-Barnes integral, visible after shifting the line of integration past the pole at $z = 1$), at $s = \tfrac{k}{2}$ (coming from $L(s - \tfrac{k}{2} + 1, \theta^k)$ and the Mellin-Barnes integral), and at $s = \tfrac{k}{2}-1$ (coming from the Mellin-Barnes integral), with no other poles for $\Re s > \tfrac{k}{2} - 2$.
The potential pole at $s = \tfrac{k}{2} + 1$ cancels, while the poles at $s = \tfrac{k}{2}$ and $s = \tfrac{k}{2} - 1$ have residues
\begin{equation*}
  \Res_{s = \frac{k}{2}} = \frac{\pi^{\frac{k}{2}}}{2\Gamma(\frac{k}{2})} \quad \text{and} \quad \Res_{s = \frac{k}{2} - 1} = \frac{\pi^{\frac{k}{2}}}{12 \Gamma(\frac{k}{2} - 1)}.
\end{equation*}

The integrand in~\eqref{eq:DsPk_continuation} has exponential decay in vertical strips from the gamma function.
Shifting the line of integration in~\eqref{eq:laplace_j1} to $k-2 + \epsilon$ for a small $\epsilon > 0$ shows that
\begin{equation*}
  \sum_{n=1}^\infty P_k(n) n^{\frac{k}{2}-1} e^{-n/X} = \frac{\pi^\frac{k}{2}\Gamma(k-1)}{2\Gamma(\frac{k}{2})} X^{k-1} +
  \frac{\pi^\frac{k}{2} \Gamma(k-2)}{12\Gamma(\frac{k}{2}-1)} X^{k-2} + O \big( X^{k-3+\epsilon} \big).
\end{equation*}
Plugging this back into~\eqref{eq:Pkexpr} completes the proof.
\end{proof}

Our proof of Theorem~\ref{thm:continuous_laplace} now follows from the three-term decomposition of $P_k(t)^2$ given in~\eqref{eq:Pk_decomp_with_floors} and Lemmas~\ref{lem:laplace_I},~\ref{lem:laplace_II}, and~\ref{lem:laplace_III}.

\section{Improving the Integrated Mean Square Estimate}\label{sec:super_jarnik}

As our second application of the main results of this paper, we translate
Theorem~\ref{thm:sharp_cutoff_theoremv2} into the same language as the mean
square estimate for the lattice point discrepancy on the sphere.
Recall that Lau~\cite{Lau1999} showed that
\[
  \int_0^X \left(P_3(t)\right)^2 dt
  =
  \frac{C_3'}{2} X^2 \log X + O\left( X^2 \right),
\]
and note that the leading constant agrees with the constant in
Theorem~\ref{thm:sharp_cutoff_theoremv2}.

We will prove the following refinement of this mean square estimate as
a corollary to Theorem~\ref{thm:sharp_cutoff_theoremv2}.

\begin{theorem}\label{thm:super_jarnik}
  There exists $\lambda >0$ such that
  \begin{equation*}
    \int_0^X P_3(t)^2 dt
    =
    \frac{C_3'}{2} X^2 \log X
    +
    \left(\frac{C_3}{2}-\frac{C_3'}{4}-\frac{\pi^2}{3}\right) X^2
    +
    O_\lambda\left(X^{2-\lambda }\right),
  \end{equation*}
  where $C_3'$ and $C_3$ are the same constants as in
  Remark~\ref{rem:smooth_coefficients}.
\end{theorem}

\begin{proof}
It suffices to prove Theorem~\ref{thm:super_jarnik} for integer $X$ as a
consequence of Heath-Brown's estimate
$P_3(n) = O(n^{21/32+\epsilon})$~\cite{HeathBrown99}.
Indeed, the contribution of the integral of $(P_3(x))^2$ over $[X,X+1]$ is
$O(X^{21/16+\epsilon})$, which is sufficiently small.

Rewrite Theorem~\ref{thm:sharp_cutoff_theoremv2} in the form
\begin{align}\label{eq:sharp_as_integral}
  \int_0^X P_3(\lfloor t \rfloor)^2 dt =\frac{C_3'}{2}X^2\log X + \left(\frac{C_3}{2} -\frac{C_3'}{4}\right) X^{2} +O_\lambda(X^{2-\lambda}).
\end{align}
As a special case of~\eqref{eq:Pk_decomp_with_floors} we have
\[
  P_3(t)^2 - P_3(\lfloor t \rfloor)^2 = 2V_3P_3(\lfloor t \rfloor)\big( \lfloor t \rfloor^\frac{3}{2} -t^\frac{3}{2}\big)+V_3^2\big(\lfloor t \rfloor^\frac{3}{2}-t^\frac{3}{2}\big)^2.
\]
The difference between~\eqref{eq:sharp_as_integral} and $\int_0^X P_3(t)^2 dt$ can therefore be written as
\begin{equation} \label{diff_integral}
  2V_3\int_0^X P_3(\lfloor t \rfloor)\big( \lfloor t \rfloor^\frac{3}{2} -t^\frac{3}{2}\big)dt+V_3^2\int_0^X\big(\lfloor t \rfloor^\frac{3}{2}-t^\frac{3}{2}\big)^2dt.
\end{equation}

The second integral in~\eqref{diff_integral} admits the approximation
\[
  V_3^2 \sum_{n=0}^{X-1} \int_n^{n+1} \big( n^\frac{3}{2} -t^\frac{3}{2}\big)^2dt = V_3^2\sum_{n=0}^{X-1} \left(\frac{3n}{4}+ O\left(1\right)\right),
\]
obtained by integrating each summand and then performing a series expansion in $n$ term-by-term.
Summing over $n \leq X-1$, we see that
\[
  V_3^2 \sum_{n=0}^{X-1} \int_n^{n+1} \big( n^\frac{3}{2} -t^\frac{3}{2}\big)^2dt = \frac{3V_3^2}{8} X^2+O\left(X \right).
\]

Now consider the first integral in~\eqref{diff_integral}.
The contribution of the integral over the range $[0,1]$ is $O(1)$.  For the rest, we again break up the integral at discontinuities and integrate termwise to obtain
\[
  2V_3\sum_{n=1}^{X-1} P_3(n) \int_{n}^{n+1} \big(n^\frac{3}{2} - t^\frac{3}{2}\big)dt =-2V_3\sum_{n=1}^{X-1} P_3(n) \left(\frac{3\sqrt{n}}{4} +O\left(\frac{1}{n^{\frac{1}{2}}}\right)\!\right).
\]
Again using Heath-Brown's bound, $P_3(n) \ll n^{\frac{21}{32} + \epsilon}$, we estimate the contribution of the error term in the series expansion above by $O(X^{\frac{1}{2}+\frac{21}{32}+\epsilon})$.

Rearranging, we write the difference between $\int_0^X P_3(t)^2 dt$ and~\eqref{eq:sharp_as_integral} as
\begin{align} \label{halfway_summary}
  \int_0^X \!\! P_3(t)^2 dt\! -\!\! \sum_{n \leq X}\! P_3(n)^2 = \frac{3V_3^2X^2}{8} -\frac{3V_3}{2}\! \sum_{n=1}^{X-1}\! P_3(n) n^\frac{1}{2} + O\!\left(X^{\frac{37}{32}+\epsilon}\right)\!.
\end{align}
It remains to estimate the partial sum $\sum_{n \leq X} P_3(n) \sqrt{n}$.

To estimate this series, we again use Perron's formula (in the form given
in~\cite[Thm 5.2 and Cor 5.3]{montgomery2006multiplicative}), giving
\begin{equation}\label{eq:perron_P3}
  \sum_{n \leq X} P_3(n) \sqrt{n}
  =
  \int_{\sigma_0 - iT}^{\sigma_0 + iT}
  D(s - \tfrac{1}{2}, P_3) \frac{X^s}{s} ds
  +
  \mathcal{R}
\end{equation}
where $\sigma_0 = 2 + \frac{5}{32} + \epsilon$ for a small $\epsilon > 0$,
$T = X^\delta$ for a small $\delta > 0$ to be specified later, and where the
remainder term can be estimated by
\begin{equation*}
  \mathcal{R} \ll \sum_{\substack{X/2 < n < 2X \\ n \neq X}}
  P_3(n) \sqrt{n} \min \Big( 1, \frac{X}{T \lvert X - n \rvert} \Big)
  +
  \frac{X^{\sigma_0}}{T} \sum_{n \geq 1} \frac{P_3(n)\sqrt n}{n^{\sigma_0}}.
\end{equation*}
By Heath-Brown's estimate, we can trivially bound the remainder term by
$\mathcal{R} \ll X^{2 + 5/32 + \epsilon - \delta} \log X$.

It follows from the decomposition~\eqref{eq:DsPk_continuation} that shifting
the line of integration in~\eqref{eq:perron_P3} to $\Re(s) = (1 + 2 \epsilon)$
passes a pole at $s = 2$ with residue $\pi X^2 / 2$ and no other poles.
It only remains to bound the growth of the shifted integral.

From~\eqref{eq:DsPk_continuation}, it is clear that $D(s, P_3)$ has polynomial
growth in vertical strips.
But unlike in \S\ref{sec:sharp_second_momentv2}, we must explicitly understand
the rate of polynomial growth.
We do this by bounding each term in the
decomposition~\eqref{eq:DsPk_continuation}.

First, we estimate the integral
\begin{equation*}
  \frac{1}{2\pi i}\!\int_{(-1+\epsilon)}
  \!\!
  L(s-1-z,\theta^3)\zeta(z)
  \frac{\Gamma(z)\Gamma(s-\frac{1}{2}-z)}{\Gamma(s-\frac{1}{2})} dz
\end{equation*}
for $\Re s = 1 + 2 \epsilon$.
Note that $L(s, \theta^3)$ is uniformly bounded in its convergent half-plane.
By the functional equation for $\zeta(z)$ and Stirling's approximation, we
estimate the integrand to be bounded by
\begin{equation*}
 (1+\lvert s \rvert)^{-2\epsilon} (1 + \lvert s - z \rvert)^{1+\epsilon} e^{-\frac{\pi}{2}(\lvert z \rvert + \lvert s - z \rvert - \lvert s \rvert)}.
\end{equation*}
When $\lvert z \rvert < \lvert s \rvert$, there is no exponential contribution and the integrand is bounded by $(1 + \lvert s \rvert)^{1-\epsilon}$ on an interval of length $O(\lvert s \rvert)$. When $\lvert z \rvert > \lvert s \rvert$, there is exponential decay in the integrand and so the contribution to the integral from this domain is $O((1+|s|)^{1-\epsilon})$.
Therefore
\begin{equation*}
  \frac{1}{2\pi i}\!\int_{(-1+\epsilon)}\!\! L(s-1-z,\theta^3)\zeta(z)\frac{\Gamma(z)\Gamma(s-\frac{1}{2}-z)}{\Gamma(s-\frac{1}{2})}dz \ll_\epsilon (1+\lvert s \rvert)^{2-\epsilon}.
\end{equation*}
Coupled with the Phragm\'en-Lindel\"of convexity estimates
\begin{align*}
  &\zeta(\tfrac{1}{2}+2\epsilon+it) \ll (1+\lvert t \rvert)^{\frac{1}{4}},
  \quad
  &&\zeta(-1+2\epsilon+it) \ll (1+\lvert t \rvert)^\frac{3}{2},
  \\
  &L(2\epsilon+it,\theta^3) \ll (1+ \lvert t \rvert)^1,
  \quad
  &&L(-1+2\epsilon+it, \theta^3) \ll (1+\lvert t \rvert)^\frac{5}{2},
\end{align*}
this implies that $D(s-\tfrac{1}{2},P_3) \ll (1+ \lvert s \rvert)^{2-\epsilon}$
on the line $\Re s = 1+2\epsilon$.

Thus the shifted integral satisfies the bound
\begin{equation*}
  \frac{1}{2\pi i} \int_{1 + 2\epsilon - iT}^{1 + 2 \epsilon + iT}
  D(s - \tfrac{1}{2}, P_3) \frac{X^s}{s} ds
  \ll
  X^{1 + 2 \delta + 2 \epsilon} \log X
\end{equation*}
and the integrals over the top and bottom portions of the rectangular
contour are bounded by
$O(X^{2 + 5/32 + \epsilon - \delta}
+
X^{1 + 2\delta + 2 \epsilon + \delta \epsilon})$.

Assembling the terms from Perron's formula, we find that
\begin{equation*}
  \sum_{n \leq X} P_3(n) \sqrt n
  =
  \frac{\pi}{2} X^2
  +
  O(X^{2 + \frac{5}{32} - \delta + \epsilon} \log X)
  +
  O(X^{1 + 2 \delta + 2 \epsilon - \delta \epsilon}).
\end{equation*}
Choosing $\delta = 37/96$ shows that
\begin{equation*}
  \sum_{n \leq X} P_3(n) \sqrt n
  =
  \frac{\pi}{2} X^2
  +
  O(X^{2 - \frac{11}{48} + 2\epsilon})
\end{equation*}
for any $\epsilon > 0$.

The theorem now follows from Theorem~\ref{thm:sharp_cutoff_theoremv2}.

\end{proof}

\appendix
\section{Gupta and Zagier}\label{sec:appendix}
% Manually adjusting appearance is frowned upon prior to publication.
% However, the default label of "Proposition Appendix  A.2" is terrible,
% and should be changed. The commands below redefine how appendix labelling
% is done.
\newcommand\appendixlabel{A}
\renewcommand{\thetheorem}{\appendixlabel.\arabic{theorem}}

To understand the diagonal part of $W_k(s)$, we must apply a Rankin--Selberg integral
to a function which is not of rapid decay.
For level one forms, Zagier~\cite{ZagierRankinSelberg} showed that one can make
some sense of Rankin--Selberg integrals with functions not of rapid decay by
truncating the standard fundamental domain at height $T$, a technique now referred to
as \emph{Zagier regularization}.
This paper requires an analogue of equation~(19) of~\cite{ZagierRankinSelberg}, which gives conditions under which the normalized Rankin--Selberg integral can be recognized as an inner product of the form $\langle F, E(\cdot, \overline{s}) \rangle$.

Performing Zagier's argument over a congruence subgroup is tedious.
In~\cite{Gupta00, gupta2000corrigendum}, Gupta shows how to generalize Zagier's
results to congruence subgroups without using Zagier normalization.
Instead, Gupta decomposes the Rankin--Selberg integral into pieces and gives
direct meromorphic continuation to the decomposition.
However, Gupta does not provide a set of conditions under which one
can recognize the Rankin--Selberg integral directly as an inner product of
the form $\langle F, E(\cdot, \overline{s}) \rangle$.

In this appendix, we show how to prove the analogous statement to equation~(19)
of~\cite{ZagierRankinSelberg} for functions not of rapid decay over congruence
subgroups, using the methods of Gupta.
We first give a brief description of the primary ingredients in Gupta's proof.
We then show how to modify Gupta's proof in order to recognize the inner
product against an Eisenstein series.
For completeness, we state this for a general congruence subgroup and adapt our
notation in place of the notation of~\cite{Gupta00}.
Note that Gupta issued a corrigendum~\cite{gupta2000corrigendum} affecting some
of the argument and notation.

Let $\mathfrak{a}_1 = \infty, \mathfrak{a}_2, \ldots, \mathfrak{a}_h$ denote
the inequivalent cusps of a congruence subgroup $\Gamma$.
As above, let $\Gamma_{\mathfrak{a}_i}$ denote the stabilizer of the cusp
$\mathfrak{a}_i$. For each cusp $\mathfrak{a}_i$, fix a matrix
$\sigma_{i} \in \mathrm{SL}_2(\mathbb{Q})$ which induces an isomorphism $\Gamma_{\mathfrak{a}_i}
\cong \Gamma_\infty$ via conjugation and satisfies $\sigma_i \infty =
\mathfrak{a}_i$.
Just as in $\Gamma_0(4)$, to each cusp we associate an Eisenstein series
\begin{equation*}
  E_{\mathfrak{a}_i}(z, s)
  =
  \sum_{\gamma \in \Gamma_{\mathfrak{a}_i} \backslash \Gamma}
  \Im(\sigma_i^{-1} \gamma z)^s.
\end{equation*}
Assemble the Eisenstein series into the vector
$\vec{E}(z, s) = ( E_{\mathfrak{a}_1}, \ldots, E_{\mathfrak{a}_h})^T$.
The Eisenstein series satisfy a functional equation
$\vec{E}(z, s) = \Phi(s) \vec{E}(z, 1-s)$, where
$\Phi(s) = (\phi_{ij}(s))_{h \times h}$ is the scattering matrix consistent with the formula \eqref{eq:scatter}.
Additional details concerning $\vec{E}(z, s)$ and $\Phi(s)$ can be found
in the discussion leading up to Theorem~4.4.2 in~\cite{kubota1973elementary}.

Let $F(z)$ denote a continuous function invariant under the action of $\Gamma$,
and let $f_{\mathfrak{a}_i}(z)$ denote the Fourier expansion of $F$ at the
cusp $\mathfrak{a}_i$, given by
\begin{equation*}
  f_{\mathfrak{a}_i}(z) = F(\sigma_i z)
  =
  \sum_{m \in \mathbb{Z}} a^{(\mathfrak{a}_i)}_m(y) e(mx).
\end{equation*}
Further, suppose that
\begin{equation*}
  f_{\mathfrak{a}_i}(z)
  =
  \psi_{\mathfrak{a}_i}(y) + O(y^{-N})
  \qquad (\text{for all } N \text{ as } \Im(z) \to \infty),
\end{equation*}
where $\psi_{\mathfrak{a}_i}$ is a function of the form
\begin{equation*}
  \psi_{\mathfrak{a}_i}(y)
  =
  \sum_{j = 1}^\ell
  \frac{c_{ij}}{n_{ij}!}y^{\alpha_{ij}} \log^{n_{ij}} y.
\end{equation*}
(We note that the corresponding equation~\cite[(3)]{Gupta00} omits the factorial.)
Denote the largest exponent of polynomial contribution
by $\Theta = \max(\alpha_{ij})$.
Finally, define the Rankin--Selberg transform of $F$ at $\mathfrak{a}_i$ as
\begin{equation*}
  \begin{split}
    R_{\mathfrak{a}_i}(F, s)
    &=
    \int_0^\infty \big(
      a_0^{(\mathfrak{a}_i)}(y) - \psi_{\mathfrak{a}_i}(y)
    \big) y^{s-2} dy
    \\
    &= \int_0^\infty \int_0^1 \big(
      F(\sigma_i z) - \psi_{\mathfrak{a}_i}(y)
    \big) y^s \frac{dx dy}{y^2}.
  \end{split}
\end{equation*}

The main theorem of~\cite{Gupta00} states that $R_{\mathfrak{a}_i}(F, s)$
has a meromorphic continuation to all $s$ in which the only potential poles
are at $s = 0, 1, \alpha_{ij}, 1-\alpha_{ij}$, and $\rho/2$, where
$\alpha_{ij}$ ranges over $i$ and the $1 \leq j \leq \ell$ in the definition of
$\psi_{\mathfrak{a}_i}(y)$, and $\rho$ ranges over the nontrivial zeros
of the Riemann zeta function.
Further, $R_{\mathfrak{a}_i}(F, s)$ satisfies a functional equation relating
it to the Rankin--Selberg transforms at the other cusps.

To prove this for a fixed cusp $\mathfrak{a}$, Gupta decomposes
$R_{\mathfrak{a}}$ into the sum
\begin{equation}\label{eq:gupta_decomp}
  R_{\mathfrak{a}}(F, s)
  =
  I_{\mathfrak{a},K}(s)
  +
  I_{\mathfrak{a}, F}(s)
  +
  I_{\mathfrak{a}, F, \psi}(s)
  +
  I_{\mathfrak{a}, \psi}(s),
\end{equation}
in which
\begin{align*}
  I_{\mathfrak{a}, K}(s)
  &=
  \iint_{K} F(z) E_{\mathfrak{a}}(z,s) \frac{dx dy}{y^2}
  \\
  I_{\mathfrak{a}, F}(s)
  &=
  \sum_{i = 1}^h \iint_\mathcal{D} F(\sigma_i z) \big(
    E_{\mathfrak{a}} (\sigma_i z, s) - e_{i \mathfrak{a}}(y,s)
  \big) \frac{dx dy}{y^2}
  \\
  I_{\mathfrak{a}, F, \psi}(s)
  &=
  \sum_{i = 1}^h \iint_\mathcal{D} \big(
    F(\sigma_i z) - \psi_{\mathfrak{a}_i}(y)
  \big) e_{i \mathfrak{a}}(y,s) \frac{dx dy}{y^2}
  \\
  I_{\mathfrak{a}, \psi}(s)
  &=
  \sum_{i = 1}^h \iint_\mathcal{D}
  \psi_{\mathfrak{a}_i} \phi_{i \mathfrak{a}}(s) y^{1 - s} \frac{dx dy}{y^2}
  -
  \iint_{\widetilde{\mathcal{D}}}
  \psi_{\mathfrak{a}} y^{s} \frac{dx dy}{y^2},
\end{align*}
where $\mathcal{D}$ is the typical fundamental
domain for $\mathrm{SL}_2(\mathbb{Z})$, $\widetilde{\mathcal{D}}$ is the
complement of $\mathcal{D}$ in the vertical strip,
$\{ x+iy \in \mathcal{H} : \lvert x \rvert \leq 1/2, \lvert z \rvert < 1 \}$,
$K$ is a compact set such that the fundamental domain
$\mathcal{D}_\Gamma = K \bigcup \sigma_i \mathcal{D}$,
and $e_{i \mathfrak{a}}(y,s)$ is the constant Fourier coefficient of
$E_{\mathfrak{a}}(\sigma_i z, s)$.
(The corrigendum~\cite{gupta2000corrigendum} to the original paper mainly
concerns the compact set $K$ in the decomposition and the corresponding
integral term $I_{\mathfrak{a}, K}$).

Most of Gupta's argument goes into proving~\eqref{eq:gupta_decomp}.
As $I_{\mathfrak{a}, F}(s)$ and $I_{\mathfrak{a}, F, \psi}(s)$ are chosen to
converge locally normally for all $s$ and $I_{\mathfrak{a},K}(s)$ is a
well-behaved integral over a compact region save for isolated poles due to the
Eisenstein series, it is straightforward to see that the remainder of the polar
behavior (and meromorphic continuation) of $R_\mathfrak{a}(F, s)$ can be
understood through $I_{\mathfrak{a}, \psi}$.
However, when $\max(\alpha_{ij}) < \frac{1}{2}$, the individual components of
$I_{\mathfrak{a}, F}$ and $I_{\mathfrak{a}, F, \psi}(s)$ converge, and it is
possible to exploit cancellation by rearranging these terms.
We now deviate from Gupta's proof.

\begin{lemma}\label{lem:gupta_three_integrals}
  The following are equivalent.
  \begin{enumerate}
    \item $\sum_{i = 1}^h \iint_\mathcal{D}
      F(\sigma_i z) E_\mathfrak{a} (\sigma_i z, s) \frac{dx dy}{y^2}$
      converges.

    \item $\sum_{i = 1}^h \iint_\mathcal{D}
      F(\sigma_i z) e_{i \mathfrak{a}}(y,s)\frac{dx dy}{y^2}$
      converges.

    \item $\sum_{i = 1}^h \iint_\mathcal{D}
      \psi_{\mathfrak{a}_i}(y) e_{i \mathfrak{a}}(y,s) \frac{dx dy}{y^2}$
      converges.
  \end{enumerate}
  Convergence refers to local normal convergence.
\end{lemma}

\begin{proof}
  The equivalence $(1) \iff (2)$ follows from the fact that
  $I_{\mathfrak{a}, F}(s)$ converges for all $s$ away from isolated poles of
  $E_\mathfrak{a}$.
  Similarly, $(2) \iff (3)$ follows from the convergence of $I_{\mathfrak{a},
  F, \psi}(s)$ for all $s$ away from isolated poles of $E_\mathfrak{a}$.
\end{proof}

\begin{lemma}\label{lem:last_gupta_integral_converges}
  Define $\Theta = \max(\alpha_{ij})$ and suppose that $\Theta < \frac{1}{2}$.
  Then the integrals
  \begin{equation*}
    \sum_{i = 1}^h \iint_\mathcal{D}
    \psi_{\mathfrak{a}_i}(y) e_{i \mathfrak{a}}(y,s) \frac{dx dy}{y^2}
  \end{equation*}
  converge locally normally if $\Theta < \Re(s) < 1 - \Theta$ and if $s$ is not
  a pole of any entry $\phi_{i\mathfrak{a}}$ of the scattering matrix.
\end{lemma}

\begin{proof}
  In terms of the entries of the scattering matrix $\Phi(s)$, the constant
  coefficient of the Eisenstein series can be written as
  \begin{equation}\label{eq:constant_term_from_scat_matrix}
    e_{i \mathfrak{a}}(y,s)
    =
    \delta_{[\mathfrak{a}_i = \mathfrak{a}]}y^s
    +
    \phi_{i \mathfrak{a}}(s) y^{1 - s}.
  \end{equation}
  Therefore it suffices to consider the convergence of the integrals
  \begin{equation*}
    I_1 = \iint_{\mathcal{D}}
    \psi_{\mathfrak{a}}(y) y^s \frac{dx dy}{y^2},
    \qquad
    I_2 = \phi_{i\mathfrak{a}}(s) \iint_{\mathcal{D}}
    \psi_{\mathfrak{a}_i}(y) y^{1-s} \frac{dx dy}{y^2}.
  \end{equation*}
  The scattering matrix element $\phi_{i\mathfrak{a}}(s)$ is independent of $y$
  and can be taken outside of the integral.
  Since poles of $\phi_{i\mathfrak{a}}(s)$ give poles of $I_2$, we suppose for the
  remainder of the proof that $s$ is not a pole of any $\phi_{i\mathfrak{a}}(s)$.

  The fundamental domain $\mathcal{D}$ can be split into the region
  $[-1/2, 1/2] \times [1, \infty)$ and the compact region
  \begin{equation}\label{def:Omega_region}
    \Omega
    :=
    \{ x + iy : \lvert x \rvert \leq 1/2; y \leq 1; \lvert z \rvert \geq 1 \}.
  \end{equation}
  Since the integrands of $I_1$ and $I_2$ are continuous and bounded, both
  integrals converge on $\Omega$.
  Further, the integrands are independent of $x$.
  Thus it suffices to consider convergence of the integrands over the halfline
  $y \geq 1$.

  Expanding and substituting $\psi_{\mathfrak{a}}$ shows that $I_1$ converges
  if and only if the integrals
  \begin{equation*}
    \int_1^\infty \sum_{j = 1}^\ell
    \frac{c_{ij}}{n_{ij}!} y^{s + \alpha_{ij} - 1} \log^{n_{ij}} y \frac{dy}{y}
    =
    \sum_{j = 1}^\ell \frac{c_{ij}}{(s + \alpha_{ij} - 1)^{n_{ij} + 1}}
  \end{equation*}
  converge (where $i$ in this expression is chosen so that
  $\mathfrak{a}_i = \mathfrak{a}$).
  The $j$th integral converges exactly when
  $\Re(s) < 1 - \alpha_{ij}$, so that for $\Re (s) < 1 - \Theta$ the above
  equality holds.
  Similarly, expanding $\psi_{\mathfrak{a}_i}$ in $I_2$ shows
  that $I_2$ converges absolutely when $\Theta < \Re(s)$ (and has poles of
  order $n_{ij}$ at $s = \alpha_{ij}$).
\end{proof}

Suppose now that the $\psi_{\mathfrak{a}_i}$ satisfy
$\Theta := \max(\alpha_{ij}) < \frac{1}{2}$, and note that
\begin{equation*}
  \langle F, E_\mathfrak{a} (\cdot, \overline{s}) \rangle
  =
  \iint_K F(z)E_\mathfrak{a}(z,s) \frac{dx dy}{y^2}
  +\iint_{\mathcal{D}}
  \sum_{i = 1}^h F(\sigma_i z) E_\mathfrak{a}(\sigma_i z, s) \frac{dx dy}{y^2}.
\end{equation*}
Then for all $s$ away from
poles of entries of the scattering matrix $\Phi(s)$ and satisfying $\Theta < \Re s < 1 - \Theta$, we can simplify the
decomposition in~\eqref{eq:gupta_decomp} using
Lemmas~\ref{lem:gupta_three_integrals}
and~\ref{lem:last_gupta_integral_converges}.
After collecting $\langle F, E_\mathfrak{a} \rangle$ from $I_{\mathfrak{a},K}$ and 
the first term in $I_{\mathfrak{a}, F}$, and cancelling the second term from
$I_{\mathfrak{a}, F}$ with the first term of $I_{\mathfrak{a}, F, \psi}$,
it follows that
\begin{equation*}
  I_{\mathfrak{a},K} + I_{\mathfrak{a}, F} + I_{\mathfrak{a}, F, \psi}
  =
  \langle F, E_\mathfrak{a} \rangle
  -
  \sum_{i = 1}^h \iint_{\mathcal{D}}
  \psi_\mathfrak{a}(y) e_{i \mathfrak{a}}(y,s) \frac{dx dy}{y^2}.
\end{equation*}
Using~\eqref{eq:constant_term_from_scat_matrix} to expand $e_{i \mathfrak{a}}(y,s)$
and adding $I_{k, \psi}$ shows that
\begin{equation*}
  R_\mathfrak{a}(F, s)
  =
  \langle F, E_\mathfrak{a} \rangle
  -
  \iint_{\mathcal{D}} \psi_\mathfrak{a}(y) y^s \frac{dx dy}{y^2}
  -
  \iint_{\widetilde{\mathcal{D}}} \psi_\mathfrak{a}(y) y^s \frac{dx dy}{y^2}.
\end{equation*}
Note that in this expression, the integral over $\widetilde{\mathcal{D}}$
is not in the region of convergence, and we are referring to the analytic
continuation of the integral there.

The integrals over $\mathcal{D}$ and $\widetilde{\mathcal{D}}$ cancel
completely.
To see this, let $\Omega$ be as in~\eqref{def:Omega_region}.
For the first integral, we see from the evaluation of $I_1$ in
Lemma~\ref{lem:last_gupta_integral_converges} that
\begin{equation*}
  \iint_{\mathcal{D}} \psi_\mathfrak{a}(y) y^s \frac{dx dy}{y^2}
  =
  \sum_{j = 1}^\ell \frac{c_{ij}}{(s + \alpha_{ij} - 1)^{n_{ij} + 1}}
  +
  \iint_{\Omega} \psi_\mathfrak{a}(y) y^s \frac{dx dy}{y^2}
\end{equation*}
On the other hand,
\begin{equation*}
  \begin{split}
    \iint_{\widetilde{\mathcal{D}}} \psi_\mathfrak{a}(y) y^s \frac{dx dy}{y^2}
    &=
    \int_0^1 \int_{-1/2}^{1/2} \psi_\mathfrak{a}(y) y^{s-1} dx \frac{dy}{y}
    -
    \iint_{\Omega} \psi_\mathfrak{a}(y) y^s \frac{dx dy}{y^2}
    \\
    &=
    -\sum_{j = 1}^\ell \frac{c_{ij}}{(s + \alpha_{ij} - 1)^{n_{ij} + 1}}
    -
    \iint_{\Omega} \psi_\mathfrak{a}(y) y^s \frac{dx dy}{y^2}.
  \end{split}
\end{equation*}
Thus the two integrals cancel.
The same proof applies to $F$ at each cusp $\mathfrak{a}$, and we have proved
the following proposition.

\begin{proposition}\label{prop:gupta_general}
  Continuing with the notation above, for $s$ satisfying
  $\Theta < \Re(s) < 1 - \Theta$, we have that
  \begin{equation*}
    R_\mathfrak{a}(F, s)
    =
    \langle F (\sigma_\mathfrak{a} \cdot), E_\mathfrak{a}(\cdot, \overline{s})\rangle.
  \end{equation*}
\end{proposition}

For our application, we take $F = \mathcal{V}$ on $\Gamma_0(4)$.
Recalling the proof of Lemma~\ref{lem:subtracting_eisenstein}, we have that
$\psi_\infty(y)$ and $\psi_0(y)$ consist only of constant multiples of
$y^{1 - \frac{k}{2}}$, and thus $\Theta = 1 - \frac{k}{2}$.
A short computation (very similar to the classic Rankin--Selberg computation)
shows that
\begin{equation}
  R_\infty(\mathcal{V}, s) = R_0(\mathcal{V}, s)
  =
  \frac{\Gamma(s + \frac{k}{2} - 1)}{(4\pi)^{s + \frac{k}{2} - 1}}
  \sum_{m = 1}^\infty \frac{r_k(m)^2}{m^{s + \frac{k}{2} - 1}}.
\end{equation}
Note that one should expect that
$R_{\infty}(\mathcal{V}, s) = R_0(\mathcal{V}, s)$,
since $\theta \big\rvert_{\sigma_0}(z) = \theta(z)$.
Applying Proposition~\ref{prop:gupta_general} gives the following Corollary.

\begin{corollary}\label{cor:diagonal_from_gupta}
  For $s$ satisfying $1 - \frac{k}{2} < \Re(s) < \frac{k}{2}$,
  we have
  \begin{equation*}
    \frac{\Gamma(s + \frac{k}{2} - 1)}{(4\pi)^{s + \frac{k}{2} - 1}}
    \sum_{m = 1}^\infty \frac{r_k(m)^2}{m^{s + \frac{k}{2} - 1}}
    =
    \langle \mathcal{V}, E_\infty(\cdot, \overline{s}) \rangle
    =
    \langle \mathcal{V}(\sigma_0 \cdot), E_0(\cdot, \overline{s}) \rangle.
  \end{equation*}
  This function has meromorphic continuation to the plane with potential
  poles at $s = \frac{k}{2}, 1, 0, 1-\frac{k}{2}$, and at zeros of $\zeta(2s)$,
  and satisfies a functional equation of shape $s \mapsto 1-s$.
\end{corollary}

\begin{remark}
  It also follows that
  \begin{equation*}
    \frac{\Gamma(s + \frac{k}{2} - 1)}{(4\pi)^{s + \frac{k}{2} - 1}}
    L(s, \theta^k \times \theta^k)
    =
    \zeta(2s) \langle \mathcal{V}, E_\infty(\cdot, \overline{s}) \rangle,
  \end{equation*}
  analogous to the relation for a typical Rankin--Selberg convolution
  $L(s, f \times g)$ between cusp forms.
  For this reason, we call $L(s, \theta^k \times \theta^k)$ the
  Rankin--Selberg convolution of $\theta^k$ and $\theta^k$.
\end{remark}

\bibliographystyle{alpha}
\bibliography{compiled_bibliography}

\end{document}